\documentclass{article}

\usepackage[hidelinks,hypertexnames=false]{hyperref}
\hypersetup{
  pdftitle={Exact oracle complexity for function-value suboptimality under weighted initial conditions},
  pdfauthor={Yoel Drori}
}
\usepackage[margin=1in]{geometry}
\usepackage{amsmath,amsthm,amssymb,mathtools}
\usepackage{mathrsfs}
\usepackage{enumitem}
\usepackage[numbers,sort&compress]{natbib}
\usepackage{microtype}
\usepackage{booktabs}
\usepackage{algorithm}
\usepackage{algpseudocode}
\usepackage{float}

\allowdisplaybreaks
\numberwithin{equation}{section}

\newtheorem{theorem}{Theorem}[section]
\newtheorem{lemma}[theorem]{Lemma}
\newtheorem{proposition}[theorem]{Proposition}
\newtheorem{corollary}[theorem]{Corollary}

\theoremstyle{definition}
\newtheorem{definition}[theorem]{Definition}

\theoremstyle{remark}
\newtheorem{remark}{Remark}[section]

\newcommand{\real}{\mathbb{R}}
\newcommand{\Rx}{R_x}
\newcommand{\Rf}{R_f}
\newcommand{\oracle}{\mathcal{O}}
\newcommand{\risk}[2]{\mathscr{R}_{#1}\bigl(#2\bigr)}
\newcommand{\Pcls}[3]{\mathcal{P}^{\mu,L}_{#1,\,#2}(\real^{#3})}
\newcommand{\Rg}{R_g}
\newcommand{\finstances}[1]{\Pcls{x}{\Rx}{#1}}

\newcommand{\Fclass}{\mathcal{F}_{\mu,L}}
\newcommand{\ip}[2]{\left\langle #1,#2\right\rangle}
\newcommand{\norm}[1]{\left\lVert #1\right\rVert}
\newcommand{\conv}{\operatorname{conv}}
\newcommand{\Span}{\operatorname{span}}
\newcommand{\Proj}{\operatorname{P}}
\newcommand{\dist}{\operatorname{dist}}
\newcommand{\argminop}{\operatorname*{argmin}}
\providecommand{\doi}[1]{\href{https://doi.org/#1}{\nolinkurl{https://doi.org/#1}}}

\begin{document}

\title{Exact oracle complexity for function-value suboptimality
under weighted initial conditions}

\author{Yoel Drori\thanks{Google Research.}}

\date{}

\maketitle

\begin{abstract}
We determine the exact worst-case function-value suboptimality after $N$ first-order oracle calls on $L$-smooth, $\mu$-strongly convex functions under a nonnegative weighted combination of the initial squared distance, function-value suboptimality, and squared gradient norm. Excluding the case where no strict improvement in function-value suboptimality can be guaranteed, the exact deterministic minimax risk in dimension $d\geq2N+1$ is characterized by the unique solution of a single scalar equation involving an $N$-step recurrence. The proof constructs an explicit hard instance for the lower bound and a constant-memory first-order method, ITEM-wf (previously called ITEM-w), whose worst-case performance matches this lower bound. On the standard initial conditions, the result recovers the exact convex bound attained by the Optimized Gradient Method of Kim and Fessler, and establishes the minimax optimality of the recently introduced ITEM-f method of Kim, Ryu, and Das Gupta for the initial function-value condition.
\end{abstract}



\section{Introduction}

Under the oracle-complexity framework of Nemirovski and Yudin, optimization
problems are studied in terms of the information that must be obtained from an
oracle to guarantee a prescribed accuracy~\cite{NemirovskiYudin1983}.  For
quadratic minimization, Krylov-information arguments yield exact finite-horizon
minimax bounds, attained by the conjugate gradient (CG)
method~\cite{Nemirovsky1991Krylov,Nemirovsky1992Operator}.  For the broader
smooth convex and smooth strongly convex classes, classical oracle-complexity
results establish the accelerated rates up to constant
factors~\cite{GuzmanNemirovski2015,ArjevaniShalevShwartzShamir2016}.
Determining the exact finite-horizon complexity in these classes likewise
requires both sides of the problem: a first-order method with a sharp
worst-case guarantee and a matching hard instance showing that no method can
improve it.

For general smooth convex and strongly convex functions, exact
finite-horizon characterizations are known only for selected criteria and
algorithmic classes.  In the smooth convex setting, the Optimized Gradient
Method (OGM)~\cite{KimFessler2016} attains the exact worst-case function-value
guarantee, with the matching oracle-complexity lower bound established
in~\cite{Drori2017Exact}.  In the strongly convex setting, Drori and
Taylor developed refined oracle-complexity lower bounds for both function-value
suboptimality and distance to the solution~\cite{DroriTaylor2022Oracle}; for
the distance criterion, the Information-Theoretic Exact Method (ITEM) attains
the lower bound exactly~\cite{taylor2023optimal}.  The same work also obtained, numerically, the method now called ITEM-f for
the function-value criterion.  Kim, Ryu, and Das Gupta recently gave an
analytic description and convergence proof of this method~\cite[Section~4.2]{KimRyuDasGupta2026}.  Thus, while the
distance criterion admits an exact finite-horizon minimax characterization in
the strongly convex setting, exactness for function-value suboptimality
remained open.

This paper closes that gap and places the result in a broader framework of
weighted initial conditions.  We determine the exact finite-horizon minimax
function-value suboptimality under nonnegative weighted combinations of the
initial squared distance, function-value suboptimality, and squared gradient
norm.  The exact coefficient is characterized by a single scalar matching
equation involving a backward recurrence.  The same recurrence produces both
an explicit hard instance and a matching constant-memory first-order method,
ITEM-wf\footnote{The method was called ITEM-w in earlier versions of this
paper.  The revised name distinguishes the weighted initial conditions from
the function-value performance criterion.  The algorithm and its guarantees
are unchanged.}.

On the lower-bound side, the construction refines that of Drori and
Taylor~\cite{DroriTaylor2022Oracle} using the zero-chain formalism of Carmon,
Duchi, Hinder, and Sidford~\cite{CarmonDuchiHinderSidford2020}.  Its additional
freedom decouples the minimizer coordinates from the directions revealed to the
method, which is what permits exact matching for general weighted initial conditions.
On the upper-bound side, the recurrence determines ITEM-wf, whose analysis uses a
potential function and smooth strongly convex interpolation.  For the pure initial
function-value condition, the exact coefficient agrees with the ITEM-f contraction
factor, so the ITEM-f upper guarantee is matched by an oracle-complexity lower bound
and its exact finite-horizon minimax optimality follows.  Together, these
constructions yield an exact finite-horizon minimax characterization for the
weighted problem class.

\paragraph{Paper organization.}
Section~\ref{S:problem} formulates the weighted problem and states the main result; Section~\ref{S:generic} develops the generic zero-chain lower-bound framework, and Sections~\ref{S:scalar-risk-coefficient}--\ref{S:lower-realization} connect it to the scalar recurrence and the resulting hard instance. Section~\ref{S:matching-method} uses the same recurrence to construct the ITEM-wf method with the matching upper bound, and Section~\ref{S:weighted-exact} combines the two constructions. The appendices collect deferred proofs and specializations of the general bound.

\section{Problem class and main results}\label{S:problem}

For $0\leq\mu<L$, let $\Fclass(\real^d)$ be the class of differentiable
functions $f:\real^d\to\real$ that are $\mu$-strongly convex and have an
$L$-Lipschitz continuous gradient.  A \emph{problem instance} is a pair
$(\oracle_f,x_0)$, where $f\in\Fclass(\real^d)$ attains its minimum,
$f_*:=\min f$, $\oracle_f(x)=(f(x),\nabla f(x))$ is the first-order oracle
of $f$, and $x_0$ is the prescribed initial point.
The three standard problem classes are
the initial-distance class
\[
    \Pcls{x}{R}{d}:=\bigl\{(\oracle_f,x_0):\;
      f\in\Fclass(\real^d),\;
      \norm{x_0-x_*}\leq R
        \text{ for some }x_*\in\argminop f\bigr\},
\]
the initial function-value suboptimality class
\[
    \Pcls{f}{R}{d}:=\bigl\{(\oracle_f,x_0):\;
      f\in\Fclass(\real^d),\;
      f(x_0)-f_*\leq R\bigr\},
\]
and the initial-gradient class
\[
    \Pcls{g}{R}{d}:=\bigl\{(\oracle_f,x_0):\;
      f\in\Fclass(\real^d),\;
      \norm{\nabla f(x_0)}\leq R\bigr\}.
\]
To treat these three initial conditions within a common framework, we introduce
a class based on their nonnegative weighted combination.  For a nonzero
nonnegative weight vector $w=(w_x,w_f,w_g)$ and $R>0$, define
\begin{equation}\label{D:mixed-class}
\begin{aligned}
 \mathcal P^{\mu,L}_{w,R}(\real^d):=\bigl\{(\oracle_f,x_0):\;&
 f\in\Fclass(\real^d),\ \text{and for some }x_*\in\argminop f,\\[-1mm]
 &w_x\norm{x_0-x_*}^2
   +\frac{2w_f}{L}\bigl(f(x_0)-f_*\bigr)
   +\frac{w_g}{L^2}\norm{\nabla f(x_0)}^2\leq R^2\bigr\}.
\end{aligned}
\end{equation}
Multiplying all weights by a scalar $s>0$ is equivalent to replacing the
budget $R$ by $R/\sqrt{s}$, so the weights may be normalized without loss of
generality.  We therefore normalize them to satisfy
\begin{equation}\label{D:weight-simplex}
    w_x,w_f,w_g\geq0,
    \qquad
    w_x+w_f+w_g=1.
\end{equation}
The three vertices of the weight simplex are denoted by
\[
    w^x:=(1,0,0),
    \qquad
    w^f:=(0,1,0),
    \qquad
    w^g:=(0,0,1).
\]
Thus
\[
    \mathcal P^{\mu,L}_{w^x,\Rx}(\real^d)=\Pcls{x}{\Rx}{d},
    \qquad
    \mathcal P^{\mu,L}_{w^f,\sqrt{2\Rf/L}}(\real^d)=\Pcls{f}{\Rf}{d},
    \qquad
    \mathcal P^{\mu,L}_{w^g,\Rg/L}(\real^d)=\Pcls{g}{\Rg}{d}.
\]

Let $\mathcal A_{N,d}$ be the deterministic first-order methods that may
depend on the parameters defining the problem class and on $(N,d)$, query
first at $x_0$, make at most $N$ oracle calls, and return an arbitrary function
of the history of oracle answers.
Finally, the \emph{minimax risk} is defined by
\begin{equation}\label{D:risk}
    \risk{N}{\mathcal P}
    :=\inf_{\mathcal M\in\mathcal A_{N,d}}
      \sup_{(\oracle_f,x_0)\in\mathcal P}
      \bigl(f(\mathcal M(\oracle_f,x_0))-f_*\bigr).
\end{equation}
Only the lower-bound construction requires $d\geq2N+1$; the scalar recurrence
is dimension-free, and the ITEM-wf upper guarantee holds in every dimension.

\subsection{Scalar characterization of the minimax risk}
\label{S:main-scalar-coefficients}

This subsection introduces the two definitions needed to state the main
result, Theorem~\ref{T:intro-main}, which expresses the minimax risk of the
weighted class through a single scalar, the risk coefficient.  Both the hard
instance of Section~\ref{S:lower-realization} and the method ITEM-wf of
Section~\ref{S:matching-method} are built from the same scalar sequence,
generated backward from the horizon by a two-dimensional recurrence.

\begin{definition}[Paired backward sequence]\label{D:backward-sequence}
Fix $N\geq1$ and $q\in[0,1)$, and write $\bar q=1-q$.  Given a radius
$r\in(0,1)$, set
\begin{equation}\label{D:backward-endpoint}
    a_N:=r,
    \qquad
    b_N:=\sqrt{\bar q}\,r.
\end{equation}
For $k=N,\ldots,2$, whenever the preceding quantities are well defined,
define the recurrence coefficients
\begin{equation}\label{D:backward-coupling}
\begin{aligned}
    c_k&:=\sqrt{(1-a_k)(1+\bar q+qa_k)},\\
    d_k&:=\bar q+qa_k,
\end{aligned}
\end{equation}
and propagate
\begin{equation}\label{D:sequence}
\begin{bmatrix}a_{k-1}\\[1mm]b_{k-1}\end{bmatrix}
=
\begin{bmatrix}
    d_k & c_k\\
    -qc_k & d_k
\end{bmatrix}
\begin{bmatrix}a_k\\[1mm]b_k\end{bmatrix}.
\end{equation}
The sequence
$((a_k,b_k))_{k=1}^N$ is the \emph{paired backward sequence associated
with $r$}.  It is \emph{ordered} if the recurrence is well defined through
$k=1$ and
\begin{equation}\label{E:sequence-order}
    1>a_1>\cdots>a_N>0,
    \qquad
    b_1>0.
\end{equation}
Where the dependence on $r$ is to be made explicit, we write $a_k(r)$ and
$b_k(r)$ for $1\leq k\leq N$.
We denote the set of radii associated with ordered sequences by
\[
    \mathcal O_N(q)
      :=\{r\in(0,1):((a_k(r),b_k(r)))_{k=1}^N\text{ is ordered}\}.
\]
\end{definition}

\begin{definition}[Risk coefficient]\label{D:risk-coefficient}
Fix $N\geq1$ and $q\in[0,1)$, write $\bar q=1-q$, and let
$w=(w_x,w_f,w_g)$ satisfy the simplex conditions~\eqref{D:weight-simplex}.
For $\eta\geq0$, set
\begin{equation}\label{D:risk-coefficient-induced-radius}
    r_{\eta,w}^2:=q+\eta\bigl(\bar q w_x-qw_f\bigr).
\end{equation}
Whenever $r_{\eta,w}^2>0$, let $r_{\eta,w}$ denote its positive square root.
We call $\eta$ a \emph{risk coefficient for $w$} if
$r_{\eta,w}\in\mathcal O_N(q)$ and the \emph{matching equation}
\begin{equation}\label{E:mixed-matching-equation}
    a_1
      +\eta(w_x+w_f)+\sqrt{\eta/\bar q}\,b_1 = 1
\end{equation}
holds, where $a_1,b_1$ are the first pair of the backward sequence associated
with $r_{\eta,w}$.
\end{definition}
At $\eta=0$ the matching equation reads $a_1=1$, which fails whenever
$r_{0,w}\in\mathcal O_N(q)$, since then $a_1<1$; hence every risk coefficient
is positive.

\begin{theorem}[Exact weighted minimax risk]\label{T:intro-main}
Fix $L>0$, $0\leq\mu<L$, $N\geq1$, put $q=\mu/L$, and let $w$ satisfy the simplex conditions~\eqref{D:weight-simplex}.  Assume that either $q>0$ or $w_x>0$.  For every
$R>0$ and every $d\geq2N+1$,
\begin{equation}\label{E:weighted-exact-minimax}
    \risk{N}{\mathcal P^{\mu,L}_{w,R}(\real^d)}
    =\frac{L}{2}\,\eta_N^\star(q,w)R^2,
\end{equation}
where $\eta_N^\star(q,w)$ denotes the unique risk coefficient characterized
above by Definition~\ref{D:risk-coefficient}.  The ITEM-wf method associated
with this risk coefficient satisfies the same upper guarantee in every dimension.
\end{theorem}
\begin{remark}[Excluded convex case]
When $q=0$ and $w_x=0$, the simplex conditions give $w_f+w_g=1$, and the
weighted initial condition contains no distance information.  Since every smooth
convex function satisfies
\[
  \|\nabla f(x_0)\|^2\leq 2L\bigl(f(x_0)-f_*\bigr),
\]
any instance with $f(x_0)-f_*\leq LR^2/2$ also satisfies
\[
  \frac{2w_f}{L}\bigl(f(x_0)-f_*\bigr)
  +\frac{w_g}{L^2}\|\nabla f(x_0)\|^2
  \leq \frac{2(w_f+w_g)}{L}\bigl(f(x_0)-f_*\bigr)
  \leq R^2.
\]
Thus the weighted class contains the usual class with bounded initial
function-value suboptimality.  On that class, no finite-horizon first-order
guarantee of any prescribed strict improvement is possible;
see~\cite[Appendix~A.2, Lemma~6]{CarmonDuchiHinderSidford2021}.
\end{remark}

\section{Convex-combination zero-chain lower bound}
\label{S:generic}

The lower-bound argument separates the construction of the hard family from the
scalar recurrence.  Here we construct a family of $L$-smooth, $\mu$-strongly convex
functions in which each oracle call can reveal only one new direction, while
retaining separate freedom over the revealed directions and the minimizer.
Section~\ref{S:lower-realization} later uses the recurrence to select a member
whose initial quantities and lower-bound value have the required relation.

\subsection{Zero-chain reduction}
\label{S:zero-chain-reduction}

A minimax lower bound must hold against every deterministic first-order
method.  The resisting-oracle reduction handles this quantifier by presenting
a zero-chain in orthonormal coordinates revealed one direction at a time.  After
$N$ oracle calls, the method can use only an $N$-dimensional revealed subspace,
so its risk is bounded below by the function's suboptimality over that
subspace.  This is the mechanism behind Nesterov's ``worst function in the
world''~\cite[Section~2.1.2]{Nesterov2004} and the earlier constructions of
Nemirovski and Yudin~\cite{NemirovskiYudin1983}.  We recall the zero-chain
definition of Carmon, Duchi, Hinder, and Sidford~\cite{CarmonDuchiHinderSidford2020}
and state the corresponding reduction for the weighted class.

We use the vectors $e_0,\ldots,e_{m-1}$ to denote the canonical unit vectors of
$\real^m$, indexed from $0$ so that $\Span\{e_0,\ldots,e_{j-1}\}$ is the
$j$-dimensional subspace reached after $j$ queries.
To formulate the zero-chain condition, denote the revealed coordinate
subspaces by
\begin{equation}\label{D:revealed-subspaces}
\begin{aligned}
    S_{-1}&:=\{0\}, \\
    S_j&:=\Span\{e_0,\ldots,e_j\},
        &&0\leq j<N.
\end{aligned}
\end{equation}

\begin{definition}[First-order zero-chain]\label{D:zero-chain}
A differentiable function $f:\real^m\to\real$ is a \emph{first-order
zero-chain of length $N$} if
\[
    x\in S_{j-1}
    \quad\Longrightarrow\quad
    \nabla f(x)\in S_j,
    \qquad 0\leq j<N.
\]
\end{definition}
This is \cite[Definition~3]{CarmonDuchiHinderSidford2020}, the
canonical-coordinate specialization of
\cite[Definition~1]{DroriTaylor2022Oracle}.  The following lemma is a
standard resisting-oracle reduction, stated for
the weighted class and the fixed-horizon risk~\eqref{D:risk}.  It differs
from the original in one respect, forced by the strongly convex setting.  The original
reduction embeds a zero-chain on $\real^m$ into the larger space $\real^d$,
and the embedded function is constant along the orthogonal complement of the
embedded subspace; when $\mu>0$, it is therefore not strongly convex on
$\real^d$.  We therefore add a quadratic along that complement, which
restores strong convexity without affecting the argument. We include the full proof for the sake of completeness.

\begin{lemma}[Weighted zero-chain reduction]\label{L:zero-chain-reduction}
Let $1\leq N\leq m$, let $w$ satisfy the simplex
conditions~\eqref{D:weight-simplex}, and let $f_z:\real^m\to\real$ be a
first-order zero-chain of length $N$ with $f_z\in\Fclass(\real^m)$.  Suppose that, for
some minimizer $x_*$ of $f_z$,
\[
    w_x\norm{x_*}^2
    +\frac{2w_f}{L}\bigl(f_z(0)-f_z^*\bigr)
    +\frac{w_g}{L^2}\norm{\nabla f_z(0)}^2
    \leq R^2.
\]
Then
\begin{equation}\label{E:zero-chain-base-mixed}
    \risk{N}{\mathcal P^{\mu,L}_{w,R}(\real^d)}
    \geq
    \inf_{x\in S_{N-1}}
       \bigl(f_z(x)-f_z^*\bigr),
    \qquad d\geq m+N.
\end{equation}
\end{lemma}

\begin{proof}
For each method we exhibit an instance in the class with initial point
$x_0=0$ on which the method's output is confined to an $N$-dimensional
subspace.  A method that stops before $N$ calls may be padded with arbitrary
additional queries after its output has been determined, without changing
that output, so it is enough to consider methods making exactly $N$ calls.

\emph{Extension to $\real^d$.}
For $U\in\real^{d\times m}$ with $U^{\top}U=I$, define
\[
    f_U(x):=f_z(U^{\top}x)+\frac{\mu}{2}\norm{(I-UU^{\top})x}^2,
    \qquad x\in\real^d.
\]
Thus $f_U$ is the orthogonal direct sum of $f_z\in\Fclass(\real^m)$ and a
$\mu$-quadratic, and therefore $f_U\in\Fclass(\real^d)$.  Moreover $Ux_*$ is
a minimizer, $f_U^*=f_z^*$, and
\[
    \norm{Ux_*}=\norm{x_*},\qquad
    f_U(0)-f_U^*=f_z(0)-f_z^*,\qquad
    \nabla f_U(0)=U\nabla f_z(0).
\]
Hence the weighted initial condition holds with the same $R$, so
$(\oracle_{f_U},0)\in\mathcal P^{\mu,L}_{w,R}(\real^d)$ for every such $U$.

\emph{Resisting oracle.}
Fix $\mathcal M\in\mathcal A_{N,d}$ making exactly $N$ calls.  Since $\mathcal M$ is deterministic,
its queries $\xi_1:=x_0=0,\xi_2,\ldots,\xi_N$ and output $\xi_{N+1}$ are
determined by the oracle answers at $\xi_1,\ldots,\xi_N$.  The columns
$u_0,\ldots,u_{m-1}$ of $U$ are the directions in $\real^d$ along which the
coordinates $e_0,\ldots,e_{m-1}$ of $f_z$ are placed, and $\xi_t$ is the
point at which the method makes its $t$th query.  At the $t$th query, only
$u_0,\ldots,u_{t-2}$ have been fixed.  The remaining columns will be chosen
orthogonal to the query $\xi_t$, so $U^{\top}\xi_t$ has support only on the already
placed coordinates.  The zero-chain property then implies that the $f_z$
component of the gradient can reveal at most the next coordinate, carried by
$u_{t-1}$.  The additional term $\mu(I-UU^{\top})\xi_t$ depends only on
$\xi_t$ and the columns already placed, so it introduces no dependence on the
still-unplaced columns.  The invariant~\eqref{E:resisting-invariant} below
formalizes this; we establish it by induction on $t$, complete $U$, and check
that the constructed answers are those of the true oracle.
Following \cite[Lemma~7 and Proposition~2]{CarmonDuchiHinderSidford2020},
we choose the columns of $U$ alongside the queries so that
\begin{equation}\label{E:resisting-invariant}
    \ip{u_j}{\xi_t}=0,
    \qquad j\geq t-1,
    \quad 1\leq t\leq N+1.
\end{equation}
For the induction, suppose $u_0,\ldots,u_{t-2}$ and $\xi_1,\ldots,\xi_t$
have been determined for some $1\leq t\leq N$.  Choose $u_{t-1}$ as a unit vector orthogonal to
$u_0,\ldots,u_{t-2}$ and $\xi_2,\ldots,\xi_t$; these $2t-2<d$ vectors leave
room.  Under~\eqref{E:resisting-invariant},
$U^{\top}\xi_t=\sum_{j\leq t-2}\ip{u_j}{\xi_t}e_j\in S_{t-2}$,
so the zero-chain property gives
$\nabla f_z(U^{\top}\xi_t)\in S_{t-1}$, and
\begin{align*}
    f_U(\xi_t)
    &=f_z(U^{\top}\xi_t)
      +\frac{\mu}{2}\Bigl\|\xi_t-\sum_{j\leq t-2}\ip{u_j}{\xi_t}u_j\Bigr\|^2,\\
    \nabla f_U(\xi_t)
    &=\sum_{j\leq t-1}\bigl(\nabla f_z(U^{\top}\xi_t)\bigr)_ju_j
      +\mu\Bigl(\xi_t-\sum_{j\leq t-2}\ip{u_j}{\xi_t}u_j\Bigr).
\end{align*}
Both depend only on $u_0,\ldots,u_{t-1}$ and $\xi_t$, so the answer at
$\xi_t$ is determined, and with it $\xi_{t+1}$.  To complete $U$ after
the $N$ steps, note that $u_0,\ldots,u_{N-1}$ and $\xi_1,\ldots,\xi_{N+1}$
are now fixed; for the remaining $m-N$ zero-chain directions choose
$u_N,\ldots,u_{m-1}$ orthonormal and orthogonal to $u_0,\ldots,u_{N-1}$ and
$\xi_2,\ldots,\xi_{N+1}$, which is possible because these $2N$ vectors leave
a complement of dimension at least $d-2N\geq m-N$.  For consistency, observe that $U^{\top}U=I$ and
\eqref{E:resisting-invariant} now holds for every $t$, so the answers used
above are the true values of $\oracle_{f_U}$ at $\xi_1,\ldots,\xi_N$; since
$\mathcal M$ is deterministic, running it on $\oracle_{f_U}$ produces
exactly these queries and the output $\xi_{N+1}$.  Finally, \eqref{E:resisting-invariant} at
$t=N+1$ gives $U^{\top}\xi_{N+1}\in S_{N-1}$.

\emph{Lower bound.}
Since $f_U\geq f_z\circ U^{\top}$ and $(\oracle_{f_U},0)$ belongs to the
class,
\[
    f_U(\xi_{N+1})-f_U^*
    \geq\inf_{x\in S_{N-1}}\bigl(f_z(x)-f_z^*\bigr),
\]
and taking the supremum over instances and the infimum over methods
in~\eqref{D:risk} proves~\eqref{E:zero-chain-base-mixed}.
\end{proof}

The lemma converts a statement about all methods into a statement about one
function: the minimax risk is at least the suboptimality of $f_z$ over the
terminal subspace, provided $f_z$ satisfies the weighted initial condition and
the zero-chain property.

\subsection{Elimination by weight transfer}
\label{S:elimination}

We first introduce the projection-based construction from which every
member of the family will be built.
Fix an integer $m\geq1$, a point $x_*\in\real^m$, and a collection $\{s_i\}_{i\in I}\subset\real^m$ for some finite index set
$I$.  Set
\[
    K:=\conv\{s_i:i\in I\},
\]
and call the points $s_i$ the \emph{sites} of $K$.
The construction uses a function whose gradient is governed by projection
onto $K$, while each site supplies a quadratic lower support.  Define
\begin{equation}\label{E:squared-distance-form}
    F(x):=\frac{L}{2}\norm{x-x_*}^2
      -\frac{L-\mu}{2}\dist^2\bigl(x,K\bigr).
\end{equation}
Using $\dist^2(x,K)=\min_{z\in K}\norm{x-z}^2$ and recalling that $\bar q=1-q$, we can rewrite $F$ as
\begin{equation}\label{D:smoothed-interpolant}
    F(x)=\frac{L}{2}\left[
      q\norm{x-x_*}^2
      +\bar q\max_{z\in K}
       \left\{\norm{x-x_*}^2-\norm{x-z}^2\right\}
    \right].
\end{equation}
The two representations expose complementary properties.  The
squared-distance form yields a smooth gradient through the projection onto
$K$, whereas the maximization form separates the $\mu$-strongly convex
quadratic term and exhibits a quadratic lower support for every site.
The maximizer in~\eqref{D:smoothed-interpolant} is $\Proj_K(x)$, where
$\Proj_K$ denotes the Euclidean projection onto $K$.  Hence the normalized
gradient is the convex combination
\begin{equation}\label{E:projection-gradient}
    \frac1L\nabla F(x)
    =q(x-x_*)+\bar q\bigl(\Proj_K(x)-x_*\bigr).
\end{equation}
Taking $z=s_i$ in~\eqref{D:smoothed-interpolant} yields a quadratic support for
each site.  For $i\in I$, write
\begin{equation}\label{D:support-quadratics}
    Q_i(x):=\frac{L}{2}\norm{x-x_*}^2
      -\frac{L-\mu}{2}\norm{x-s_i}^2.
\end{equation}
Thus
\begin{equation}\label{E:support-ineq}
    F(x)\geq Q_i(x),
    \qquad x\in\real^m,\quad i\in I.
\end{equation}
The next lemma establishes the class membership and the minimizer condition.

\begin{lemma}
\label{L:smoothed-interpolant}
Let $K\subset\real^m$ be the convex hull of a finite family of sites
$(s_i)_{i\in I}$, let $x_*\in\real^m$, and let $F$ be
given by~\eqref{E:squared-distance-form}.  Then:
\begin{enumerate}[label=(\roman*),leftmargin=2.2em]
    \item $F\in\Fclass(\real^m)$, with gradient given
    by~\eqref{E:projection-gradient};
    \item if $x_*\in K$, then $x_*$ is a minimizer of
    $F$, with $F(x_*)=0$.
\end{enumerate}
\end{lemma}

\begin{proof}
For part~(i), $\mu$-strong convexity follows from~\eqref{D:smoothed-interpolant}.
The Moreau-envelope identity
$\nabla\frac12\dist^2(\cdot,K)=I-\Proj_K$~\cite{Moreau1965} gives
\eqref{E:projection-gradient}; nonexpansiveness of $\Proj_K$ then gives the
$L$-Lipschitz gradient bound.  Part~(ii) follows from
$\dist(x,K)\leq\norm{x-x_*}$ when $x_*\in K$.
\end{proof}

In the family constructed next, $K$ will contain the stage points
$s_0,\ldots,s_N$ together with a distinguished point $s_*$ corresponding to
the minimizer.

To establish that the constructed function is a zero-chain, it suffices to show
that, for every $x$ attainable after $j$ queries,
$\Proj_K(x)\in\conv\{s_0,\ldots,s_j\}$.  We prove the stronger property
that the sites $s_{j+1},\ldots,s_N$ and $s_*$ have zero coefficient in every
convex representation of $\Proj_K(x)$ by the sites.
The next lemma gives a criterion for excluding one site $s_k$: it suffices to find a point $z\in K$ such that, if a convex representation of $\Proj_K(x)$ assigned positive weight to $s_k$, transferring a small amount of that weight from $s_k$ to $z$ would decrease the distance to $x$. The stated inner-product condition guarantees precisely this, contradicting optimality of the projection.

\begin{lemma}[Elimination by weight transfer]\label{L:face-exchange}
Let $K=\conv\{s_i:i\in I\}\subset\real^m$ for a finite family of sites,
let $k\in I$, and let $x\in\real^m$.  Suppose there exists $z\in K$ with
$z\neq s_k$ such that
\begin{equation}\label{E:transfer-condition}
    \ip{s_k-z}{s_i-x}\geq0
       \quad\text{for every }i\in I.
\end{equation}
Then every convex representation of $\Proj_K(x)$ by the sites has zero
coefficient at $k$: i.e., if
\[
    \Proj_K(x)=\sum_{i\in I}\alpha_i s_i,
    \qquad
    \alpha_i\geq0,\quad \sum_{i\in I}\alpha_i=1,
\]
then $\alpha_k=0$.
\end{lemma}

\begin{proof}
Set $p:=\Proj_K(x)$ and fix a convex representation
$p=\sum_{i\in I}\alpha_i s_i$.

Since $z$ is a convex combination of the sites, \eqref{E:transfer-condition}
gives
$\ip{s_k-z}{z-x}\geq0$, hence
$\ip{s_k-z}{s_k-x}\geq\norm{s_k-z}^2$.
We get
\begin{equation}\label{E:fe-lower}
  \ip{s_k-z}{p-x}=\sum_{i\in I}\alpha_i\ip{s_k-z}{s_i-x}
  \;\geq\;\alpha_k\ip{s_k-z}{s_k-x}\;\geq\;\alpha_k\norm{s_k-z}^2.
\end{equation}

Suppose $\alpha_k>0$; we exhibit a point of $K$ closer to $x$ than $p$.
For $\varepsilon\in(0,1]$, the point
\[
    p_\varepsilon
    :=\sum_{i\neq k}\alpha_is_i+(1-\varepsilon)\alpha_ks_k+\varepsilon\alpha_kz
    =p-\varepsilon\alpha_k(s_k-z)
\]
is a convex combination of points of $K$, hence lies in $K$.
By~\eqref{E:fe-lower},
\[
    \norm{p_\varepsilon-x}^2
    =\norm{p-x}^2-2\varepsilon\alpha_k\ip{s_k-z}{p-x}
      +\varepsilon^2\alpha_k^2\norm{s_k-z}^2
    \leq\norm{p-x}^2-\varepsilon(2-\varepsilon)\alpha_k^2\norm{s_k-z}^2
    <\norm{p-x}^2,
\]
using $z\neq s_k$; this contradicts $p=\Proj_K(x)$.  Hence $\alpha_k=0$.
\end{proof}

\subsection{A family of hard functions}
\label{S:hard-family}

We now specialize this construction to a scalar-parametrized family.  The scalars
$D_k$, $\gamma_k$, and $\delta_k$ specify, respectively, the minimizer
coordinates, gradients, and chain points, and each scalar tuple $\mathcal C$
determines a candidate member $F_{\mathcal C}$.  The admissibility conditions
introduced below identify the hard subfamily: for an admissible $\mathcal C$,
every unrevealed site can be excluded from the projection at the corresponding
stage, so that $F_{\mathcal C}$ reveals at most one new coordinate at a time.

\begin{definition}[Chain-parametrized candidate function]
\label{D:chain-function}
Fix $N\geq1$ and $q\in[0,1)$, and let
\[
  \mathcal C
  :=\bigl((D_k)_{k=0}^N,(\gamma_k)_{k=0}^N,
          (\delta_k)_{k=0}^{N-1}\bigr)
\]
be a tuple of scalars.  Work in $\real^{N+1}$ with canonical basis
$e_0,\ldots,e_N$, and define
\begin{equation}\label{D:generic-vertices}
\begin{aligned}
    \widetilde x_*&:=-\sum_{\ell=0}^ND_\ell e_\ell,\\
    \widetilde g_i&:=L\gamma_ie_i,
       &&0\leq i\leq N,\\
    \widetilde x_i&:=-\sum_{\ell=0}^{i-1}\delta_\ell e_\ell,
       &&0\leq i\leq N.
\end{aligned}
\end{equation}
For $I:=\{0,\ldots,N,*\}$, define the site offsets
\begin{equation}\label{D:generic-transformed-sites}
\begin{aligned}
    t_i&:=L^{-1}(\widetilde g_i-\mu\widetilde x_i)
       =q\sum_{\ell=0}^{i-1}\delta_\ell e_\ell+\gamma_ie_i,
       &&0\leq i\leq N,\\
    t_*&:=-q\widetilde x_*=q\sum_{\ell=0}^ND_\ell e_\ell,
\end{aligned}
\end{equation}
and put
\begin{equation}\label{D:generic-sites}
    s_i:=\frac{\widetilde x_*+t_i}{\bar q},
    \qquad i\in I,
    \qquad
    K_{\mathcal C}:=\conv\{s_i:i\in I\}.
\end{equation}
Let $F_{\mathcal C}$ denote the instance of~\eqref{E:squared-distance-form}
determined by $\widetilde x_*$ and $K_{\mathcal C}$.
\end{definition}

Thus every site is obtained from its offset by the same affine map,
\begin{equation}\label{E:generic-site-expansion}
    \bar q s_i-\widetilde x_*=t_i,
    \qquad i\in I.
\end{equation}
For $0\leq i\leq N$, the choice of $t_i$ is precisely the requirement that
the support quadratic~\eqref{D:support-quadratics} attain the prescribed
gradient at the corresponding chain point: since
$\nabla Q_i(x)=\mu x+L(\bar qs_i-\widetilde x_*)$, we have
$\nabla Q_i(\widetilde x_i)=\widetilde g_i$.  At the distinguished index,
$t_*=-q\widetilde x_*$ gives $s_*=\widetilde x_*$, so the corresponding
support quadratic is stationary at the minimizer.

The following definition collects conditions that will be sufficient
for establishing that $F_{\mathcal C}$ is a zero-chain.

\begin{definition}[Admissible chain]\label{D:admissible-chain}
An \emph{admissible chain} is a tuple
\[
  \mathcal C
  =\bigl((D_k)_{k=0}^N,(\gamma_k)_{k=0}^N,
          (\delta_k)_{k=0}^{N-1}\bigr)
\]
whose scalars and associated coefficients satisfy the following conditions.
The scalars are required to satisfy
\begin{equation}\label{E:generic-signs}
\begin{aligned}
  D_k&>0, &&0\leq k\leq N,\\
  \gamma_k&>q\delta_k,\qquad \delta_k\geq0,
  &&0\leq k<N,\\
  0<\gamma_N&<D_N,
\end{aligned}
\end{equation}
and
\begin{equation}\label{E:generic-balance}
  (\gamma_k-q\delta_k)(D_k-\gamma_k)
    \geq\gamma_{k+1}D_{k+1},
  \qquad 0\leq k<N,
\end{equation}
whereas the associated coefficients defined by $\pi_0:=1$ and recursively
\begin{equation}\label{E:generic-pi-recurrence}
  \pi_{k+1}:=\frac{\pi_k\gamma_k-qD_k}
                  {\gamma_k-q\delta_k},
  \qquad 0\leq k<N-1,
\end{equation}
are required to satisfy
\begin{align}
  \pi_0\geq \pi_1\geq\cdots\geq \pi_{N-1}&>0,
  \label{E:generic-pi-monotone}\\
  \bigl(\pi_{N-1}\gamma_{N-1}-qD_{N-1}\bigr)
   \bigl(D_{N-1}-\gamma_{N-1}\bigr)
   &\geq qD_N^2.
  \label{E:generic-terminal-admissibility}
\end{align}
\end{definition}

Establishing the zero-chain condition requires two applications of the
weight-transfer criterion in Lemma~\ref{L:face-exchange}.  At stage $j$, each
unrevealed chain site can be eliminated by transferring any positive weight on
that site to the revealed site $s_j$, using the chain
conditions.  The coefficients
$\pi_0,\ldots,\pi_j$ determine a replacement point in
$\conv\{s_0,\ldots,s_j\}$ for the minimizer site $s_*$; transferring weight
from $s_*$ to this point again decreases the distance to the query point.
The admissibility conditions in the last definition
ensure that the hypotheses of Lemma~\ref{L:face-exchange} hold for both
transfers.

\begin{proposition}[Zero-chain property]\label{P:revealed-face}
Let $\mathcal C$ be an admissible chain. For every $0\leq j<N$,
$x\in S_{j-1}$ implies
\begin{equation}\label{E:revealed-face-projection}
\begin{aligned}
    \Proj_{K_{\mathcal C}}(x)&\in\conv\{s_0,\ldots,s_j\},
    \\
    \nabla F_{\mathcal C}(x)&\in S_j.
\end{aligned}
\end{equation}
Thus $F_{\mathcal C}$ is a first-order zero-chain of length $N$.
\end{proposition}

The proof, given in Appendix~\ref{A:revealed-face-proof}, applies
Lemma~\ref{L:face-exchange} once per excluded site.

The lower bound corresponding to an admissible chain now follows directly from
the zero-chain reduction.  It is summarized by four scalar quantities: $\Phi$
measures the terminal suboptimality forced on the final revealed subspace
$S_{N-1}$, while $B_x,B_f,B_g$ are the three initial quantities of the associated function.  We
first define these quantities and record their geometric interpretation; the
theorem below then assembles the construction.

\begin{definition}[Certificate and initial-condition data]\label{D:chain-certificate-data}
For a tuple
\[
  \mathcal C
  =\bigl((D_k)_{k=0}^N,(\gamma_k)_{k=0}^N,
          (\delta_k)_{k=0}^{N-1}\bigr),
\]
define its \emph{certificate value} $\Phi$ and \emph{initial-condition data} $(B_x,B_f,B_g)$ by
\begin{equation}\label{D:chain-quantities}
    \Phi:=D_N^2-
    \frac{(D_N-\gamma_N)^2
          +q\sum_{\ell=0}^{N-1}(D_\ell-\delta_\ell)^2}{\bar q},
\end{equation}
and
\begin{equation}\label{E:chain-quantities-explicit}
\begin{aligned}
    B_x&:=\sum_{\ell=0}^ND_\ell^2,\\
    B_f&:=D_0^2-\frac{(D_0-\gamma_0)^2+q\sum_{\ell=1}^ND_\ell^2}{\bar q},\\
    B_g&:=\gamma_0^2.
\end{aligned}
\end{equation}
\end{definition}

\begin{lemma}[Certificate and initial-condition data]
\label{L:certificate-data}
Let $\mathcal C$ be an admissible chain, let $F_{\mathcal C}$ be its associated
function from Definition~\ref{D:chain-function}, and let
$(\Phi,B_x,B_f,B_g)$ be its certificate and initial-condition data.  Then:
\begin{enumerate}[label=(\roman*),leftmargin=2.2em]
    \item on the final revealed subspace $S_{N-1}$,
    \begin{equation}\label{E:generic-restricted-value}
        \inf_{x\in S_{N-1}}
        F_{\mathcal C}(x)
        \geq Q_N(\widetilde x_N)
        =\frac{L}{2}\Phi;
    \end{equation}
    \item at the starting point $x_0=0$,
    \begin{equation}\label{E:generic-initial-quantities}
    \begin{aligned}
        B_x&=\norm{\widetilde x_*}^2,\\
        B_f&=\frac{2}{L}F_{\mathcal C}(0),\\
        B_g&=\frac{1}{L^2}\norm{\nabla F_{\mathcal C}(0)}^2,
    \end{aligned}
    \end{equation}
    and these three quantities are positive.
\end{enumerate}
\end{lemma}

The proof is a direct computation from Definition~\ref{D:chain-function}
and Proposition~\ref{P:revealed-face}; it is given in
Appendix~\ref{A:certificate-data-proof}.

\begin{proposition}[Lower bound from an admissible chain]
\label{P:generic-hard-function}
Let $\mathcal C$ be an admissible chain, let $F_{\mathcal C}$ be its associated
function from Definition~\ref{D:chain-function}, and let
$(\Phi,B_x,B_f,B_g)$ denote its certificate and initial-condition data.  Then:
\begin{enumerate}[label=(\roman*),leftmargin=2.2em]
    \item $F_{\mathcal C}\in\Fclass(\real^{N+1})$, with minimizer $\widetilde x_*$ and minimum value $0$;
    \item $F_{\mathcal C}$ is a first-order zero-chain of length $N$ in the
    sense of Definition~\ref{D:zero-chain}.
\end{enumerate}
Consequently, for every $w$ satisfying the simplex
conditions~\eqref{D:weight-simplex} and every $d\geq2N+1$,
\begin{equation}\label{E:generic-risk-mixed}
    \risk{N}{\mathcal P^{\mu,L}_{w,R}(\real^d)}
    \geq\frac{L}{2}\Phi,
\end{equation}
where $R>0$ is defined by
\begin{equation}\label{D:generic-mixed-budget}
    R^2=w_xB_x+w_fB_f+w_gB_g.
\end{equation}
\end{proposition}

\begin{proof}
For (i), Lemma~\ref{L:smoothed-interpolant}(i) gives
$F_{\mathcal C}\in\Fclass(\real^{N+1})$, and since $\widetilde x_*=s_*\in K_{\mathcal C}$
by~\eqref{D:generic-sites}, Lemma~\ref{L:smoothed-interpolant}(ii) shows
that $\widetilde x_*$ is a minimizer.  Part (ii) is the second claim
of~\eqref{E:revealed-face-projection} in Proposition~\ref{P:revealed-face},
which by~\eqref{D:revealed-subspaces} is the zero-chain condition of
Definition~\ref{D:zero-chain}.

For~\eqref{E:generic-risk-mixed}, apply
Lemma~\ref{L:zero-chain-reduction} to $f_z=F_{\mathcal C}$ with $m=N+1$: its
hypotheses hold by (i), (ii), and~\eqref{D:generic-mixed-budget} together
with Lemma~\ref{L:certificate-data}(ii), and its conclusion combined with
Lemma~\ref{L:certificate-data}(i) gives, for every $d\geq m+N=2N+1$,
\[
    \risk{N}{\mathcal P^{\mu,L}_{w,R}(\real^d)}
    \geq\inf_{x\in S_{N-1}}F_{\mathcal C}(x)
    \geq\frac{L}{2}\Phi.
    \qedhere
\]
\end{proof}

The theorem shows that every admissible chain determines a certified hard
function.  Selecting a suitable member of this hard subfamily therefore
reduces to constructing one admissible chain and evaluating its four scalar
certificate quantities.

\begin{remark}[The Drori--Taylor construction]
Setting $\delta_k=D_k$ for $0\leq k<N$ gives $\pi_k=1$ for every $k$, so
the replacement point used to eliminate the minimizer site reduces to the
single revealed site $s_j$.  Under this specialization, the
admissibility conditions reduce to those of
\cite[Theorem~4]{DroriTaylor2022Oracle}, and the resulting construction
recovers their function-value lower bound.  Thus the additional freedom
$\delta_k\neq D_k$ is precisely what decouples the minimizer coordinates from
the revealed chain in the present family.
\end{remark}

\section{Scalar recurrence and matching equation}
\label{S:scalar-risk-coefficient}

Theorem~\ref{T:intro-main} characterizes the minimax risk implicitly: a
candidate risk coefficient determines a recurrence parameter, the recurrence
parameter generates a backward sequence, and feasibility asks that this
sequence stay well defined while the matching equation holds.  This section
shows that such a coefficient exists.  We first prove that feasibility of the
backward sequence is stable under perturbation of the parameter and passes to
limits under a uniform margin; a continuation argument along the
recurrence-feasible set then forces a sign change of the matching residual,
hence a root.  The section ends by extending the recurrence through the initial boundary, in the form used by both constructions.

\subsection{Properties of the backward recurrence}
\label{S:recurrence-interface}

The backward recurrence has a simple quadratic structure that will be used
repeatedly below.  Each recurrence step preserves a weighted quadratic form,
which yields an invariant along every well-defined suffix of the sequence.
The same calculation also gives the inverse recurrence, allowing the sequence
to be viewed equivalently in either direction.

\begin{lemma}
\label{L:recurrence-interface}
Let $N\geq1$, $q\in[0,1)$, and $r\in(0,1)$.  Run the backward recurrence
of Definition~\ref{D:backward-sequence} from its terminal values for as long
as it is well defined.  For every $j\in\{1,\ldots,N-1\}$ for which
$(a_j,b_j)$ is well defined,
\begin{equation}\label{E:recurrence-interface-isometry}
    d_{j+1}^2+qc_{j+1}^2=1.
\end{equation}
Consequently, for every $j\in\{1,\ldots,N\}$ for which $(a_j,b_j)$ is well
defined,
\begin{equation}\label{E:recurrence-interface-factorizations}
    qa_j^2+b_j^2=r^2.
\end{equation}
Whenever $1\leq j<N$ and $(a_j,b_j)$ is well defined, the inverse recurrence is
\begin{equation}\label{E:recurrence-interface-inverse}
\begin{aligned}
\begin{bmatrix}a_{j+1}\\[1mm]b_{j+1}\end{bmatrix}
=
\begin{bmatrix}
    d_{j+1} & -c_{j+1}\\
    qc_{j+1} & d_{j+1}
\end{bmatrix}
\begin{bmatrix}a_j\\[1mm]b_j\end{bmatrix}.
\end{aligned}
\end{equation}
\end{lemma}

\begin{proof}
We first identify the one-step isometry behind the recurrence.  The coefficient
definitions~\eqref{D:backward-coupling} give
\[
\begin{aligned}
    d_{j+1}^2+qc_{j+1}^2
      &=d_{j+1}^2+q(1-a_{j+1})(1+d_{j+1})\\
      &=d_{j+1}\bigl(d_{j+1}+q(1-a_{j+1})\bigr)
        +q(1-a_{j+1})\\
      &=d_{j+1}+q(1-a_{j+1})=1,
\end{aligned}
\]
which proves~\eqref{E:recurrence-interface-isometry}.  This identity makes
the matrix in~\eqref{D:sequence} preserve the quadratic form $qu^2+v^2$:
\[
\begin{aligned}
    qa_j^2+b_j^2
      &=q\bigl(d_{j+1}a_{j+1}+c_{j+1}b_{j+1}\bigr)^2
        +\bigl(-qc_{j+1}a_{j+1}+d_{j+1}b_{j+1}\bigr)^2\\
      &=\bigl(d_{j+1}^2+qc_{j+1}^2\bigr)
        \bigl(qa_{j+1}^2+b_{j+1}^2\bigr)\\
      &=qa_{j+1}^2+b_{j+1}^2.
\end{aligned}
\]
The terminal values satisfy $qa_N^2+b_N^2=r^2$, so iterating the last
identity proves~\eqref{E:recurrence-interface-factorizations}.

The same isometry identity also gives the reverse step.  The matrix in
\eqref{D:sequence} has determinant one, and hence its inverse is precisely the
matrix in~\eqref{E:recurrence-interface-inverse}.
\end{proof}

The remaining results establish the stability properties needed for the
continuation argument.  For ordered sequences, the recurrence preserves
positivity of the second coordinates; together with the quadratic invariant,
this yields stability under perturbations and under limits that remain
uniformly inside the ordered region.  For $q>0$, we then identify the explicit
radius $r=\sqrt q$ as a canonical interior point from which the continuation
argument of the next subsection can begin.

\begin{corollary}[Second-coordinate positivity]
\label{C:recurrence-positive}
Every ordered paired backward sequence satisfies
$b_k>0$ for $1\leq k\leq N$.
\end{corollary}

\begin{proof}
The case $k=1$ is part of the ordered-sequence condition.  For
$2\leq k\leq N$, the first row of~\eqref{D:sequence} and
\eqref{D:backward-coupling} give
\[
    c_kb_k=a_{k-1}-d_ka_k
      =(a_{k-1}-a_k)+qa_k(1-a_k)>0.
\]
The ordered-sequence condition gives $0<a_k<1$, hence $c_k>0$, and therefore
$b_k>0$.
\end{proof}
\begin{lemma}
\label{L:backward-stability}
Let $N\geq1$ and $q\in[0,1)$.  Then:
\begin{enumerate}[label=(\roman*),leftmargin=2.2em]
    \item The set $\mathcal O_N(q)$ is open, and $a_k(r)$ and $b_k(r)$
    are continuous on $\mathcal O_N(q)$ for $1\leq k\leq N$.
    \item Let $(r_j)_{j\geq1}$ be a sequence in $\mathcal O_N(q)$ with
    $r_j\to r_*$.  If
    \[
        \liminf_{j\to\infty}
        \min\{1-a_1(r_j),\,b_1(r_j)^2\}>0,
    \]
    then $r_*\in\mathcal O_N(q)$.
\end{enumerate}
\end{lemma}

\begin{proof}
For part~(i), fix $r_0\in\mathcal O_N(q)$.  By backward induction from
\eqref{D:backward-endpoint}, each coefficient $a_k(r)$ and $b_k(r)$ is
continuous near $r_0$: indeed, \eqref{E:sequence-order} gives
$0<a_{k+1}(r_0)<1$, so every radicand in~\eqref{D:backward-coupling} is
positive at $r_0$.  The radicands and the inequalities in
\eqref{E:sequence-order} therefore persist in a neighborhood of $r_0$.
Thus $\mathcal O_N(q)$ is open, and the coefficients are continuous on it.

For part~(ii), choose
$0<\varepsilon<\liminf_{j\to\infty}
\min\{1-a_1(r_j),b_1(r_j)^2\}$.  Then, for all sufficiently large $j$,
$1-a_1(r_j)\geq\varepsilon$ and $b_1(r_j)^2\geq\varepsilon$.
For such $j$, the invariant gives
$r_j^2=qa_1(r_j)^2+b_1(r_j)^2\geq\varepsilon$, and hence
$r_j\geq\sqrt\varepsilon$.  Together with
$r_j\in\mathcal O_N(q)$ and the terminal identity $a_N(r)=r$ in
\eqref{D:backward-endpoint}, this gives
\[
    \sqrt\varepsilon\leq r_j=a_N(r_j)\leq a_k(r_j)\leq a_1(r_j)
    \leq1-\varepsilon,
    \qquad 1\leq k\leq N,
\]
for all sufficiently large $j$, and hence
$\sqrt\varepsilon\leq r_*\leq1-\varepsilon<1$.  Along the same tail,
$a_k(r_j)\leq a_1(r_j)$, the invariant, and
Corollary~\ref{C:recurrence-positive} give
\[
    b_k(r_j)^2=r_j^2-qa_k(r_j)^2
      \geq r_j^2-qa_1(r_j)^2=b_1(r_j)^2\geq\varepsilon.
\]
Since $b_k(r_j)>0$, we have $b_k(r_j)\geq\sqrt{\varepsilon}$.  Since
$a_{k+1}(r_j)\leq1-\varepsilon$, every radicand
in~\eqref{D:backward-coupling} is at least $\varepsilon(1+\bar q)$ along
the sequence.  Starting from $a_N(r_*)=r_*$, this uniform bound permits the
recurrence at $r_*$ to be defined backward through all indices.  Once this
limiting recurrence is defined, continuity of each recurrence step gives, by
backward induction, $a_k(r_j)\to a_k(r_*)$ and $b_k(r_j)\to b_k(r_*)$.
Passing to the limit in the preceding bounds gives
$\sqrt{\varepsilon}\leq a_k(r_*)\leq1-\varepsilon$ and
$b_k(r_*)\geq\sqrt{\varepsilon}$ for every $k$.  The first row of the inverse
recurrence~\eqref{E:recurrence-interface-inverse} therefore gives
\[
\begin{aligned}
    a_k(r_*)-a_{k+1}(r_*)
      &=(1-d_{k+1}(r_*))a_k(r_*)+c_{k+1}(r_*)b_k(r_*)\\
      &=q(1-a_{k+1}(r_*))a_k(r_*)+c_{k+1}(r_*)b_k(r_*)>0,
      \qquad 1\leq k<N.
\end{aligned}
\]
Together with the well-definedness established above,
\eqref{E:sequence-order} shows that $r_*\in\mathcal O_N(q)$.
\end{proof}

\begin{lemma}[The radius $r=\sqrt q$]\label{L:sqrtq-membership}
Let $N\geq1$ and $q\in(0,1)$.  Then
\[
    \sqrt q\in\mathcal O_N(q).
\]
\end{lemma}

\begin{proof}
Set $y_k:=b_k/\sqrt q$.  At the terminal index,
\[
    (a_N,y_N)=(\sqrt q,\sqrt{\bar q}),
\]
so this pair lies in the open first quadrant of the unit circle.

Suppose that, for some $k<N$,
\[
    a_{k+1}=\cos\theta,
    \qquad
    y_{k+1}=\sin\theta,
    \qquad
    0<\theta<\frac{\pi}{2}.
\]
Since $d_{k+1}=\bar q+qa_{k+1}$,
\[
    d_{k+1}-a_{k+1}
      =\bar q(1-a_{k+1})>0,
    \qquad
    1-d_{k+1}
      =q(1-a_{k+1})>0.
\]
Thus $a_{k+1}<d_{k+1}<1$, so
$d_{k+1}=\cos\phi$ for some $0<\phi<\theta$.  Moreover,
$c_{k+1}>0$, and~\eqref{E:recurrence-interface-isometry} gives
\[
    \sqrt q\,c_{k+1}=\sin\phi.
\]
Dividing the second row of~\eqref{D:sequence} by $\sqrt q$ therefore yields
\[
\begin{bmatrix}a_k\\ y_k\end{bmatrix}
=
\begin{bmatrix}
    \cos\phi&\sin\phi\\
    -\sin\phi&\cos\phi
\end{bmatrix}
\begin{bmatrix}
    \cos\theta\\ \sin\theta
\end{bmatrix}
=
\begin{bmatrix}
    \cos(\theta-\phi)\\
    \sin(\theta-\phi)
\end{bmatrix}.
\]
Since $0<\theta-\phi<\theta<\pi/2$,
\[
    0<a_{k+1}<a_k<1,
    \qquad
    y_k>0.
\]
Backward induction therefore gives
$1>a_1>\cdots>a_N>0$ and $b_1>0$, so
$\sqrt q\in\mathcal O_N(q)$.
\end{proof}

\subsection{Matching equation and existence}
\label{S:risk-coefficient-matching}

The matching equation~\eqref{E:mixed-matching-equation} is the boundary
condition imposed by the weight vector $w$ at the opposite end of the backward
sequence from its horizon.  Its residual, defined first, measures how far the
first pair of an ordered sequence is from satisfying this condition.  The
lemma below supplies margins that prevent a limiting radius from leaving
$\mathcal O_N(q)$.

\begin{definition}[Matching residual]
\label{D:mixed-matching-residual}
Fix $w=(w_x,w_f,w_g)$ satisfying the simplex
conditions~\eqref{D:weight-simplex}.  For $\eta\geq0$ such that
$r_{\eta,w}\in\mathcal O_N(q)$, the \emph{matching residual} is
\begin{equation}\label{E:mixed-matching-residual}
    H_w(\eta)
      :=1-a_1-\eta(w_x+w_f)-\sqrt{\frac{\eta}{\bar q}}\,b_1,
\end{equation}
where $a_1,b_1$ are the first pair of the backward sequence associated with
$r_{\eta,w}$.
\end{definition}

By Lemma~\ref{L:backward-stability}(i) and the continuity of
$\eta\mapsto r_{\eta,w}$, the set
$\{\eta\geq0:r_{\eta,w}\in\mathcal O_N(q)\}$ is
relatively open in $[0,\infty)$ and $H_w$ is continuous on it.  By
Definition~\ref{D:risk-coefficient}, the risk coefficients are exactly the
positive zeros of $H_w$ in this set.  The next lemma provides the margin
needed to continue the recurrence while the residual is nonnegative.

\begin{lemma}[Margin under a nonnegative residual]
\label{L:existence-margin}
Let $N\geq1$, $q\in[0,1)$, let $w$ satisfy the simplex
conditions~\eqref{D:weight-simplex}, and let $\eta>0$ satisfy
$r_{\eta,w}\in\mathcal O_N(q)$ and $H_w(\eta)\geq0$.  Then
\begin{equation}\label{E:existence-uniform-bounds}
    \min\{1-a_1,\,b_1^2\}
      \geq\eta\left(w_x+\frac{q^2}{\bar q}\right)>0.
\end{equation}
\end{lemma}

\begin{proof}
Since $r_{\eta,w}\in\mathcal O_N(q)$, we have $0<a_1<1$ and
$b_1>0$.  From $H_w(\eta)\geq0$ and $w_f\geq0$,
\[
    1-a_1
      =H_w(\eta)+\eta(w_x+w_f)
        +\sqrt{\frac{\eta}{\bar q}}\,b_1
      \geq
        \eta w_x+\sqrt{\frac{\eta}{\bar q}}\,b_1.
\]
On the other hand, the invariant
\eqref{E:recurrence-interface-factorizations} and
\eqref{D:risk-coefficient-induced-radius} give
\[
    r_{\eta,w}^2=q+\eta(\bar q w_x-qw_f),
\]
and therefore
\[
\begin{aligned}
    b_1^2
      &=r_{\eta,w}^2-qa_1^2\\
      &\geq r_{\eta,w}^2-qa_1\\
      &=q(1-a_1)+\eta(\bar q w_x-qw_f)\\
      &=qH_w(\eta)+\eta w_x
        +q\sqrt{\frac{\eta}{\bar q}}\,b_1\\
      &\geq
        \eta w_x+q\sqrt{\frac{\eta}{\bar q}}\,b_1.
\end{aligned}
\]
Since $b_1>0$, the last inequality implies
$b_1\geq q\sqrt{\eta/\bar q}$.  Substituting this bound into the two
preceding inequalities yields
\[
    b_1^2
      \geq\eta\left(w_x+\frac{q^2}{\bar q}\right)
\]
and
\[
    1-a_1
      \geq\eta\left(w_x+\frac{q}{\bar q}\right)
      \geq\eta\left(w_x+\frac{q^2}{\bar q}\right),
\]
where the last inequality uses $0\leq q<1$.  This proves
\eqref{E:existence-uniform-bounds}.

It remains to verify that the common lower bound is positive.  If $q>0$ this
is immediate.  If $q=0$, then $r_{\eta,w}\in\mathcal O_N(0)$ and
\eqref{D:risk-coefficient-induced-radius} give
$r_{\eta,w}^2=\eta w_x>0$, so $w_x>0$.
\end{proof}

\begin{proposition}[Existence of a risk coefficient]
\label{P:mixed-existence}
Let $q\in[0,1)$, $N\geq1$, and let $w$ satisfy the simplex
conditions~\eqref{D:weight-simplex}.  Assume that either $q>0$ or $w_x>0$.
Then a risk coefficient for $w$ in the sense of
Definition~\ref{D:risk-coefficient} exists.
\end{proposition}

\begin{proof}
If $q=0$, then $w_x>0$ by assumption, and
Proposition~\ref{P:convex-weighted-risk-coefficient} gives a risk coefficient.
Hence assume that $q>0$.

\emph{The continuation component.}
By Lemma~\ref{L:sqrtq-membership}, $\sqrt q\in\mathcal O_N(q)$.  Since
$r_{0,w}=\sqrt q$, \eqref{E:mixed-matching-residual} gives
$H_w(0)=1-a_1(\sqrt q)>0$.  Let $J_w$ denote the connected component of
\[
    \{\eta\geq0:r_{\eta,w}\in\mathcal O_N(q)\}
\]
containing zero.  By Lemma~\ref{L:backward-stability}(i) and the continuity
of $\eta\mapsto r_{\eta,w}$, this set is relatively open in $[0,\infty)$.
Hence
\[
    J_w=[0,\eta_*)
    \qquad\text{for some }\eta_*\in(0,\infty].
\]

We show that the residual must become negative before the component ends.
We argue by contradiction: if the residual remained nonnegative on $J_w$,
Lemma~\ref{L:existence-margin} would keep the paired sequence uniformly away
from the boundary of $\mathcal O_N(q)$, and
Lemma~\ref{L:backward-stability} would extend $J_w$ beyond its endpoint.

\emph{Continuation past the endpoint.}
Suppose, toward a contradiction, that $H_w(\eta)\geq0$ for all $\eta\in J_w$.
Lemma~\ref{L:existence-margin} applies at every $\eta\in(0,\eta_*)$.
Since $1-a_1<1$, \eqref{E:existence-uniform-bounds} gives
\[
    \eta<\left(w_x+\frac{q^2}{\bar q}\right)^{-1},
\]
so $\eta_*<\infty$.

Choose a sequence $0<\eta_j\uparrow\eta_*$, put
$r_j:=r_{\eta_j,w}$, and set
\[
    \varepsilon:=\eta_1\left(w_x+\frac{q^2}{\bar q}\right)>0.
\]
By~\eqref{E:existence-uniform-bounds},
\[
    1-a_1(r_j)\geq\varepsilon,
    \qquad
    b_1(r_j)^2\geq\varepsilon,
    \qquad j\geq1.
\]
The invariant~\eqref{E:recurrence-interface-factorizations} therefore gives
\[
    r_j^2=qa_1(r_j)^2+b_1(r_j)^2\geq\varepsilon.
\]
Since $r_j^2=q+\eta_j(\bar q w_x-qw_f)$, passing to the limit yields
\[
    q+\eta_*(\bar q w_x-qw_f)\geq\varepsilon>0.
\]
Thus $r_{\eta_*,w}$ is well defined, and continuity of the positive square
root gives $r_j\to r_{\eta_*,w}$.  Lemma~\ref{L:backward-stability}(ii) now
shows that $r_{\eta_*,w}\in\mathcal O_N(q)$.  Part~(i) of the same lemma,
together with the continuity of $\eta\mapsto r_{\eta,w}$, then gives
$r_{\eta,w}\in\mathcal O_N(q)$ for every $\eta$ in some interval
$[\eta_*,\eta_*+\varepsilon')$.  This contradicts the definition of $J_w$
and $\eta_*$.  Hence the residual cannot remain nonnegative on $J_w$.

\emph{Sign change and matching root.}
There is therefore some $\eta_+\in J_w$ with $H_w(\eta_+)<0$.  Since
$[0,\eta_+]\subseteq J_w$, the continuity of $H_w$ noted above applies on
this interval, and $H_w(0)>0>H_w(\eta_+)$.  The intermediate value theorem gives
$\eta\in(0,\eta_+)$ with $H_w(\eta)=0$, and by
Definition~\ref{D:risk-coefficient} this zero is a risk coefficient.
\end{proof}

\subsection{Initial coupling}
\label{S:recurrence-matched}

The recurrence is propagated from the horizon, whereas both the lower and
upper constructions connect it to the weighted data at the initial point.
For $\eta\geq0$ with $r_{\eta,w}\in\mathcal O_N(q)$, let
$((a_k,b_k))_{k=1}^N$ be the associated paired backward sequence.  We extend
the recurrence coefficients to index $1$ by setting
\begin{equation}\label{D:recurrence-interface-coefficients}
\begin{aligned}
    c_1&:=b_1+\sqrt{\bar q\eta},\\
    d_1&:=\bar q+qa_1,
\end{aligned}
\end{equation}
and continue the matrix step once more by setting
\begin{equation}\label{D:recurrence-boundary}
\begin{bmatrix}a_0\\[1mm]b_0\end{bmatrix}
:=
\begin{bmatrix}
    d_1 & c_1\\
    -qc_1 & d_1
\end{bmatrix}
\begin{bmatrix}a_1\\[1mm]b_1\end{bmatrix}.
\end{equation}
The boundary pair is not part of the backward sequence itself; it packages
the weighted initial condition in the same matrix form used by the interior
recurrence.

\begin{lemma}[Weighted initial identities]
\label{L:recurrence-interface-matched}
Let $N\geq1$, $q\in[0,1)$, let $w$ satisfy the simplex
conditions~\eqref{D:weight-simplex}, and let $\eta>0$ satisfy
$r_{\eta,w}\in\mathcal O_N(q)$.
Then the boundary quantities above satisfy
\begin{subequations}\label{E:recurrence-interface-initial}
\begin{align}
    a_0&=1-\eta w_f-\bar q H_w(\eta),
    \label{E:recurrence-interface-a0}\\
    b_0&=\bar q\,b_1-qa_1\sqrt{\bar q\eta},
    \label{E:recurrence-interface-b0}\\
    c_1b_1&=(1-a_1)(1+qa_1)
       -\eta w_f-\bar q H_w(\eta),\notag\\
    c_1^2&=(1-a_1)(1+d_1)
       +\eta\bigl(\bar q w_g-w_f\bigr)-2\bar q H_w(\eta).
    \label{E:recurrence-interface-c1-square}
\end{align}
\end{subequations}
\end{lemma}

The boundary matrix step also extends two recurrence identities used by both
the hard-instance construction and its certificate calculation.
\begin{lemma}[Boundary continuation identities]
\label{L:recurrence-boundary-identities}
Let $N\geq1$, $q\in[0,1)$, let $w$ satisfy the simplex
conditions~\eqref{D:weight-simplex}, and let $\eta>0$ satisfy
$r_{\eta,w}\in\mathcal O_N(q)$.
Use the associated backward sequence and boundary quantities from
Definition~\ref{D:backward-sequence} and
\eqref{D:recurrence-interface-coefficients}--\eqref{D:recurrence-boundary}.
Then, for $0\leq k<N$,
\begin{equation}\label{E:recurrence-boundary-cross}
    a_kb_{k+1}-a_{k+1}b_k
      =r_{\eta,w}^2c_{k+1}.
\end{equation}
Moreover, for $0\leq k<N$,
\begin{equation}\label{E:recurrence-boundary-product}
    (1-a_{k+1})(1+qa_{k+1})-(1-a_k)=b_{k+1}c_{k+1}.
\end{equation}
\end{lemma}

The proofs of the two lemmas are given in
Appendix~\ref{A:boundary-identities}.

\section{The hard instance}
\label{S:lower-realization}

This section carries out the second stage of the lower-bound argument.
Section~\ref{S:generic} certifies a lower bound for every admissible chain;
here we construct one particular chain from a candidate $\eta$ with
$r_{\eta,w}\in\mathcal O_N(q)$ and evaluate its certificate.  The chain is
chosen so that its certificate can be evaluated explicitly in terms of the matching
residual; rescaling to the prescribed budget then gives the
lower bound of Theorem~\ref{T:intro-main}.

\subsection{Certificate decomposition}
\label{S:certificate-decomposition}

Proposition~\ref{P:generic-hard-function} bounds the risk below by the
certificate value $\Phi$ of any admissible chain, and the hard instance
should make this bound as large as possible relative to the budget
$R^2=w_xB_x+w_fB_f+w_gB_g$.  The following lemma expresses the gap
$\eta(w_xB_x+w_fB_f+w_gB_g)-\Phi$, for an arbitrary tuple of chain
scalars, as a sum of squares and weighted slacks in the
chain condition~\eqref{E:generic-balance}, together with one term
proportional to the matching residual.  For an admissible chain every
term except the last is nonnegative, so the gap is smallest when all
squares and slacks vanish; the candidate chain of the next subsection is
the tuple for which they do.

\begin{lemma}[Certificate decomposition]\label{L:certificate-decomposition}
Let $\eta>0$ satisfy $r_{\eta,w}\in\mathcal O_N(q)$, and use the associated
recurrence and boundary data from Sections~\ref{S:recurrence-interface}
and~\ref{S:recurrence-matched}.  Let $(D_k)_{k=0}^N$,
$(\gamma_k)_{k=0}^N$, $(\delta_k)_{k=0}^{N-1}$ be real numbers, and let
$\Phi,B_x,B_f,B_g$ be the
quantities~\eqref{D:chain-quantities}--\eqref{E:chain-quantities-explicit}
formed from these numbers.  Then
\begin{equation}\label{E:certificate-decomposition}
\begin{aligned}
    \eta(w_xB_x+w_fB_f+w_gB_g)-\Phi
      &=\frac{1}{\bar q}\Bigg\{
        \sum_{k=0}^{N-1}(b_{k+1}D_k-c_{k+1}\gamma_k)^2
       +(\gamma_N-a_ND_N)^2\\
      &\qquad\qquad+q\sum_{k=0}^{N-1}
        \bigl(\delta_k-a_{k+1}D_k-(1-a_{k+1})\gamma_k\bigr)^2\\
      &\qquad\qquad+2\sum_{k=0}^{N-1}(1-a_{k+1})
        \Bigl[(\gamma_k-q\delta_k)(D_k-\gamma_k)
              -D_{k+1}\gamma_{k+1}\Bigr]\Bigg\}\\
      &\qquad-2H_w(\eta)\gamma_0(D_0-\gamma_0).
\end{aligned}
\end{equation}
\end{lemma}

The proof is given in Appendix~\ref{A:certificate-decomposition-proof}.

\begin{corollary}\label{C:certificate-bound}
Let $\eta>0$ satisfy $r_{\eta,w}\in\mathcal O_N(q)$, and let $\mathcal C$
be an admissible chain with certificate and initial-condition data $(\Phi,B_x,B_f,B_g)$.
Then
\[
    \Phi\leq\eta(w_xB_x+w_fB_f+w_gB_g)+2H_w(\eta)\gamma_0(D_0-\gamma_0),
\]
with equality if and only if every nonnegative term on the right-hand side
of~\eqref{E:certificate-decomposition}, apart from the residual term, is zero.
\end{corollary}

\begin{proof}
For an admissible chain, each bracketed chain-condition slack in
\eqref{E:certificate-decomposition} is nonnegative by
\eqref{E:generic-balance}, and $0<a_{k+1}<1$ by
\eqref{E:sequence-order}.  Hence every term on the
right-hand side other than the last is nonnegative.
\end{proof}

\subsection{Candidate chain}
\label{S:candidate-chain}

We now define the candidate chain as the tuple that makes every square and
every slack in~\eqref{E:certificate-decomposition} vanish, so that the
certificate is controlled by the matching residual alone.  Vanishing of the
first-line squares fixes the ratios $\gamma_k/D_k$.  We choose the
displacements $\delta_k$ to make the second-line squares vanish.  Once the
ratios and displacements are fixed, only the coordinate scales $D_k$ remain.
Imposing equality in the chain condition then determines these scales recursively,
leaving the single free scale $D_0>0$.

Let $\eta>0$ satisfy $r_{\eta,w}\in\mathcal O_N(q)$ and
$H_w(\eta)\geq0$, and take the boundary quantities $a_0,b_0,c_1$ from
Section~\ref{S:recurrence-matched}.  Fix $D_0>0$.  Recursively, once $D_k$
has been defined, set
\begin{equation}\label{D:mixed-primal-data}
\begin{aligned}
    \gamma_k&:=\frac{b_{k+1}}{c_{k+1}}D_k,\\
    \delta_k&:=a_{k+1}D_k+(1-a_{k+1})\gamma_k,
\end{aligned}
\qquad 0\leq k<N,
\end{equation}
and then impose equality in the chain condition by setting
\begin{equation}\label{D:mixed-primal-scales}
    D_{k+1}
      :=\left[
          \frac{c_{k+2}}{b_{k+2}}
          (\gamma_k-q\delta_k)(D_k-\gamma_k)
        \right]^{1/2},
    \qquad 0\leq k<N-1.
\end{equation}
At the terminal step, set
\begin{equation}\label{D:mixed-primal-terminal}
\begin{aligned}
    D_N&:=\left[
        \frac{(\gamma_{N-1}-q\delta_{N-1})(D_{N-1}-\gamma_{N-1})}
             {a_N}
      \right]^{1/2},\\
    \gamma_N&:=a_ND_N.
\end{aligned}
\end{equation}
Whenever this recursion is well defined, denote its output by
\[
    \mathcal C_{w,\eta,D_0}
    :=\bigl((D_k)_{k=0}^N,(\gamma_k)_{k=0}^N,
            (\delta_k)_{k=0}^{N-1}\bigr),
\]
and let $(\Phi,B_x,B_f,B_g)$ denote its certificate and initial-condition data
from Definition~\ref{D:chain-certificate-data}.  Thus the paired recurrence
fixes the ratios $\gamma_k/D_k$ and the displacements in
\eqref{D:mixed-primal-data}, while \eqref{D:mixed-primal-scales} propagates
only the scale $D_k$.  The lemma below shows that the recursion is well
defined and yields an admissible chain.

\begin{lemma}[Well-definedness and admissibility of the candidate chain]
\label{L:candidate-admissible}
Let $N\geq1$, $q\in[0,1)$, let $w$ satisfy the simplex
conditions~\eqref{D:weight-simplex}, and let $\eta>0$ satisfy
$r_{\eta,w}\in\mathcal O_N(q)$ and $H_w(\eta)\geq0$.  For every $D_0>0$:
\begin{enumerate}[label=(\roman*),leftmargin=2.2em]
    \item the construction~\eqref{D:mixed-primal-data}--\eqref{D:mixed-primal-terminal}
    is well defined, its scalars satisfy~\eqref{E:generic-signs}, and
    \begin{equation}\label{E:mixed-primal-chain-equality}
        D_{k+1}\gamma_{k+1}
          =(\gamma_k-q\delta_k)(D_k-\gamma_k),
        \qquad 0\leq k<N;
    \end{equation}
    \item its output $\mathcal C_{w,\eta,D_0}$ is admissible.
\end{enumerate}
\end{lemma}

\begin{proof}
\emph{Proof of (i).}
For $1\leq k<N$, the first row of~\eqref{D:sequence} and the definition of
$c_{k+1}$ give
\[
\begin{aligned}
    c_{k+1}(c_{k+1}-b_{k+1})
      &=c_{k+1}^2-a_k+d_{k+1}a_{k+1}\\
      &=(1-a_k)+\bar q(1-a_{k+1})>0.
\end{aligned}
\]
Hence $0<b_{k+1}<c_{k+1}$.  At the boundary,
$c_1=b_1+\sqrt{\bar q\eta}>b_1>0$, while $0<a_N<1$.  Therefore
\[
    0<\frac{b_{k+1}}{c_{k+1}}<1,
    \qquad 0\leq k<N,
\]
and $0<a_N<1$.

We also need the boundary sign $b_0>0$.  Equation~\eqref{E:recurrence-interface-b0} and Lemma~\ref{L:existence-margin} give
\[
    b_1^2\geq\eta\left(w_x+\frac{q^2}{\bar q}\right)
      \geq\frac{q^2\eta}{\bar q},
\]
so $b_1\geq q\sqrt{\eta/\bar q}$.  Hence
\begin{equation}\label{E:recurrence-boundary-b0-positive}
    b_0>0.
\end{equation}
Indeed, for $q>0$,
$b_0\geq q(1-a_1)\sqrt{\bar q\eta}>0$, while for $q=0$ we have
$b_0=b_1>0$.

The definitions in~\eqref{D:mixed-primal-data} also give, for $0\leq k<N$,
\[
\begin{aligned}
    c_{k+1}(\gamma_k-q\delta_k)
      &=\bigl(d_{k+1}b_{k+1}-qc_{k+1}a_{k+1}\bigr)D_k
       =b_kD_k,
\end{aligned}
\]
where the last equality is the second row of~\eqref{D:recurrence-boundary}
for $k=0$ and of~\eqref{D:sequence} for $1\leq k<N$.

We now proceed recursively.  Suppose $D_k>0$ for some $0\leq k<N-1$.  By
\eqref{D:mixed-primal-data} and the ratio bound above,
$0<\gamma_k<D_k$ and, since $0<a_{k+1}<1$,
\[
    \gamma_k<\delta_k<D_k.
\]
The preceding identity, together with
\eqref{E:recurrence-boundary-b0-positive} and $b_k>0$ for
$1\leq k\leq N$ from Corollary~\ref{C:recurrence-positive}, gives
$\gamma_k-q\delta_k>0$.  Since $D_k-\gamma_k>0$ and
$b_{k+2},c_{k+2}>0$, the radicand in~\eqref{D:mixed-primal-scales} is
positive, so $D_{k+1}>0$.  Starting from $D_0>0$, induction defines the
construction through $D_{N-1}$.

The same argument at $k=N-1$ gives
$0<\gamma_{N-1}<\delta_{N-1}<D_{N-1}$ and
$\gamma_{N-1}-q\delta_{N-1}>0$.  Since $a_N>0$, the first line of
\eqref{D:mixed-primal-terminal} therefore defines $D_N>0$, and its second
line together with $0<a_N<1$ gives $0<\gamma_N<D_N$.  Thus the entire
recursion is well defined and all sign conditions~\eqref{E:generic-signs}
hold.

For $0\leq k<N-1$, \eqref{D:mixed-primal-data} and
\eqref{D:mixed-primal-scales} give
\[
\begin{aligned}
    D_{k+1}\gamma_{k+1}
      &=\frac{b_{k+2}}{c_{k+2}}D_{k+1}^2\\
      &=(\gamma_k-q\delta_k)(D_k-\gamma_k).
\end{aligned}
\]
At $k=N-1$, the same equality follows directly from the two lines of
\eqref{D:mixed-primal-terminal}.  This proves
\eqref{E:mixed-primal-chain-equality} and hence the chain
condition~\eqref{E:generic-balance}.  This completes the proof of~(i).

\emph{Proof of (ii).}
It remains to verify the conditions involving the coefficients $\pi_k$ and
the terminal admissibility condition in Definition~\ref{D:admissible-chain}.
Using $c_{k+1}\gamma_k=b_{k+1}D_k$ from
\eqref{D:mixed-primal-data} and the identity
$c_{k+1}(\gamma_k-q\delta_k)=b_kD_k$ established in part~(i), the
recurrence~\eqref{E:generic-pi-recurrence} is equivalent to
\begin{equation}\label{E:mixed-weight-coupled-recurrence}
    b_k\pi_{k+1}=b_{k+1}\pi_k-qc_{k+1},
    \qquad 0\leq k<N-1.
\end{equation}
Put $r:=r_{\eta,w}$.  By~\eqref{E:recurrence-interface-a0} and the radius definition
\eqref{D:risk-coefficient-induced-radius},
\[
    r^2-qa_0
      =\bar q\bigl(\eta w_x+qH_w(\eta)\bigr).
\]
Thus, to solve the coupled recurrence for $\pi_k$, introduce the
nonnegative boundary ratio
\[
    \lambda
      :=\frac{r^2-qa_0}{b_0}
       =\frac{\bar q\bigl(\eta w_x+qH_w(\eta)\bigr)}{b_0}\geq0.
\]
Multiplying~\eqref{E:mixed-weight-coupled-recurrence} by $r^2$ and using
\eqref{E:recurrence-boundary-cross},
\[
\begin{aligned}
 b_k\bigl(r^2\pi_{k+1}-qa_{k+1}\bigr)
   &=r^2b_{k+1}\pi_k-qr^2c_{k+1}-qa_{k+1}b_k\\
   &=b_{k+1}\bigl(r^2\pi_k-qa_k\bigr),
\end{aligned}
\]
for $0\leq k<N-1$.  Thus
$(r^2\pi_k-qa_k)/b_k$ is independent of $k$.  Since $\pi_0=1$ and
$r^2-qa_0=\lambda b_0$, this constant is $\lambda$, and hence
\begin{equation}\label{E:mixed-weight-explicit}
    \pi_k=\frac{qa_k+\lambda b_k}{r^2},
    \qquad 0\leq k<N.
\end{equation}

This formula gives positivity directly.  If $q>0$, then $a_k>0$ by
\eqref{E:sequence-order}; if $q=0$, $r_{\eta,w}\in\mathcal O_N(0)$ gives
$r^2=\eta w_x>0$, so $w_x>0$ and hence
$\lambda=\eta w_x/b_0>0$.  Since $b_0>0$ by
\eqref{E:recurrence-boundary-b0-positive} and $b_k>0$ for $1\leq k\leq N$
by Corollary~\ref{C:recurrence-positive}, in either case
\[
    \pi_k>0,
    \qquad 0\leq k<N.
\]

It remains to verify monotonicity.  Rearranging
\eqref{E:generic-pi-recurrence},
\[
    \pi_{k+1}-\pi_k
      =\frac{q(\pi_k\delta_k-D_k)}
             {\gamma_k-q\delta_k},
    \qquad 0\leq k<N-1.
\]
The denominator is positive.  Starting from $\pi_0=1$, if $\pi_k\leq1$,
then
\[
    \pi_k\delta_k\leq\delta_k<D_k,
\]
so $\pi_{k+1}\leq\pi_k$.  Induction therefore gives
\[
    1=\pi_0\geq\pi_1\geq\cdots\geq\pi_{N-1}>0.
\]
This proves~\eqref{E:generic-pi-monotone}.

Finally, we verify the terminal condition.  Using again
$c_N\gamma_{N-1}=b_ND_{N-1}$ and
$c_N(\gamma_{N-1}-q\delta_{N-1})=b_{N-1}D_{N-1}$ gives
\[
    \frac{\pi_{N-1}\gamma_{N-1}-qD_{N-1}}
         {\gamma_{N-1}-q\delta_{N-1}}
      =\frac{b_N\pi_{N-1}-qc_N}{b_{N-1}}.
\]
Substituting~\eqref{E:mixed-weight-explicit} and then using
\eqref{E:recurrence-boundary-cross} at $k=N-1$,
\[
\begin{aligned}
    \frac{\pi_{N-1}\gamma_{N-1}-qD_{N-1}}
         {\gamma_{N-1}-q\delta_{N-1}}
      &=\frac{
        q a_{N-1}b_N+\lambda b_{N-1}b_N-q r^2c_N
      }{r^2b_{N-1}}\\
      &=\frac{
        q\bigl(a_{N-1}b_N-r^2c_N\bigr)
        +\lambda b_{N-1}b_N
      }{r^2b_{N-1}}\\
      &=\frac{qa_N+\lambda b_N}{r^2}.
\end{aligned}
\]
By the chain equality~\eqref{E:mixed-primal-chain-equality} and
$\gamma_N=a_ND_N$,
\[
    (\gamma_{N-1}-q\delta_{N-1})(D_{N-1}-\gamma_{N-1})
      =a_ND_N^2.
\]
Consequently,
\[
\begin{aligned}
&\bigl(\pi_{N-1}\gamma_{N-1}-qD_{N-1}\bigr)
   (D_{N-1}-\gamma_{N-1})-qD_N^2\\
&\quad=
 \left[
   \frac{a_N(qa_N+\lambda b_N)}{r^2}-q
 \right]D_N^2\\
&\quad=\frac{\lambda a_Nb_N}{r^2}D_N^2\\
&\quad=\sqrt{\bar q}\,\lambda D_N^2\geq0,
\end{aligned}
\]
where the last two equalities use
$a_N=r$ and $b_N=\sqrt{\bar q}\,r$.
This is~\eqref{E:generic-terminal-admissibility}, completing the proof of~(ii).
\end{proof}

\subsection{Lower bound from the candidate chain}
\label{S:chain-realization}

By Lemma~\ref{L:candidate-admissible}, the candidate chain is
admissible, so Proposition~\ref{P:generic-hard-function} applies to it and
bounds the risk below by $\frac{L}{2}\Phi$ for the budget
$R^2=w_xB_x+w_fB_f+w_gB_g$.  By Lemma~\ref{L:certificate-decomposition},
$\Phi$ differs from $\eta R^2$ only by the residual term, so it remains to
fix the scale: for fixed $\eta$, the candidate chain has a single free scale
$D_0$; rescaling $D_0$ rescales $\Phi,B_x,B_f,B_g$ by the square of the
same factor, so the initial budget can be normalized to any prescribed $R^2$.

\begin{theorem}[Lower bound from the candidate chain]
\label{T:mixed-realization}
Fix $L>0$, $0\leq\mu<L$, $N\geq1$, put $q=\mu/L$, and let $w$
satisfy the simplex conditions~\eqref{D:weight-simplex}.  Let $\eta>0$
satisfy $r_{\eta,w}\in\mathcal O_N(q)$ and $H_w(\eta)\geq0$.  For every
$D_0>0$,
the certificate and
initial-condition data of the candidate chain $\mathcal C_{w,\eta,D_0}$
satisfy
\begin{equation}\label{E:mixed-realized-lower}
    \Phi=\eta\bigl(w_xB_x+w_fB_f+w_gB_g\bigr)+2H_w(\eta)\gamma_0(D_0-\gamma_0).
\end{equation}
Consequently, for every $R>0$ and every $d\geq2N+1$,
\begin{equation}\label{E:ordered-radius-lower-bound}
    \risk{N}{\mathcal P^{\mu,L}_{w,R}(\real^d)}
      \geq\frac{L}{2}\,\eta R^2.
\end{equation}
\end{theorem}

\begin{proof}
\emph{Certificate identity.}
The definition of $\gamma_k$ in~\eqref{D:mixed-primal-data} gives
$c_{k+1}\gamma_k=b_{k+1}D_k$, while
\eqref{D:mixed-primal-data}--\eqref{D:mixed-primal-terminal} give
$\delta_k=a_{k+1}D_k+(1-a_{k+1})\gamma_k$ for $0\leq k<N$ and
$\gamma_N=a_ND_N$.  Lemma~\ref{L:candidate-admissible} shows that
the chain condition holds with equality.  These three identities annihilate
the three square terms in
\eqref{E:certificate-decomposition}, while the chain equality annihilates
every chain-condition slack.  Hence equality holds in
Corollary~\ref{C:certificate-bound}, giving~\eqref{E:mixed-realized-lower}.

\emph{Scaling and the lower bound.}
The scalar data $a_k,b_k,c_k$ depend only on $\eta$ and not on $D_0$.
The definitions \eqref{D:mixed-primal-data}--\eqref{D:mixed-primal-terminal}
therefore show
that multiplying $D_0$ by a factor $t>0$ multiplies every
$D_k,\gamma_k,\delta_k$ by the same factor $t$.  Hence
$\Phi,B_x,B_f,B_g$ are all homogeneous of degree two in $D_0$.

Since the candidate chain is admissible by Lemma~\ref{L:candidate-admissible},
Lemma~\ref{L:certificate-data}(ii) gives $B_x,B_f,B_g>0$.  Since $w$
belongs to the simplex, the weighted quantity
$w_xB_x+w_fB_f+w_gB_g$ is therefore strictly positive.  By the quadratic
homogeneity just noted, we may rescale $D_0$ so that
\[
    w_xB_x+w_fB_f+w_gB_g=R^2.
\]
For this scaling,~\eqref{E:mixed-realized-lower} gives $\Phi\geq\eta R^2$,
since $H_w(\eta)\geq0$ and admissibility gives $0<\gamma_0<D_0$.
Applying the generic hard-function lower bound~\eqref{E:generic-risk-mixed}
yields
\[
    \risk{N}{\mathcal P^{\mu,L}_{w,R}(\real^d)}
      \geq\frac{L}{2}\Phi
      \geq\frac{L}{2}\,\eta R^2,
    \qquad d\geq2N+1.
\]
This proves~\eqref{E:ordered-radius-lower-bound}.
\end{proof}

By Corollary~\ref{C:certificate-bound}, no admissible chain attains a
larger certificate value than $\eta(w_xB_x+w_fB_f+w_gB_g)$ at a risk
coefficient, so among the lower bounds obtainable from
Proposition~\ref{P:generic-hard-function} the candidate chain gives the
largest.

\section{ITEM-wf and the matching upper bound}
\label{S:matching-method}

This section constructs ITEM-wf and proves the upper bound that matches the
hard-instance lower bound of Section~\ref{S:lower-realization} at a risk
coefficient.  Its analysis is a standard
potential-function (or Lyapunov-function)
argument~\cite{BansalGupta2019}, with the potential decrease certified by
smooth strongly convex interpolation inequalities as in~\cite{TaylorBach2019}.

Given $\eta>0$ with $r_{\eta,w}\in\mathcal O_N(q)$, take the paired
backward sequence
$((a_k,b_k))_{k=1}^N$ associated with $r_{\eta,w}$ and the boundary
quantities $a_0,b_0,c_1,d_1$ from Section~\ref{S:recurrence-matched};
together with the recurrence coefficients $c_k,d_k$, $2\leq k\leq N$,
from~\eqref{D:backward-coupling}, these are the scalar data used below.
Algorithm~\ref{A:itemw} specifies ITEM-wf for these coefficients.

\begin{algorithm}
\caption{ITEM-wf associated with $\eta$}\label{A:itemw}
\begin{algorithmic}[1]
\Require Initial point $x_0$, horizon $N\geq1$, class parameters $L,q$, and
coefficient data $(a_k)_{k=1}^N$, $(b_k)_{k=1}^N$,
$(c_k)_{k=1}^N$ associated with $\eta$.
\State $z_0\gets x_0$, $v_{-1}\gets x_0$
\For{$k=0,\ldots,N-1$}
    \State $\displaystyle
    x_k\gets z_k+
      \frac{a_{k+1}}{\bar q(1-a_{k+1})}(z_k-v_{k-1})$
    \State $x_k^+\gets x_k-L^{-1}\nabla f(x_k)$
    \State $\displaystyle
    h_k\gets x_k-x_k^+-q(x_k-z_k)$
    \State $\displaystyle
    v_k\gets v_{k-1}-\frac{c_{k+1}}{b_{k+1}}h_k$
    \State $\displaystyle
    z_{k+1}\gets(1-a_{k+1})x_k^++a_{k+1}v_k$
\EndFor
\State \Return $z_N$
\end{algorithmic}
\end{algorithm}

Because $r_{\eta,w}\in\mathcal O_N(q)$, \eqref{E:sequence-order} and
Corollary~\ref{C:recurrence-positive} show that every denominator in the
algorithm is positive.  Since $z_0=v_{-1}=x_0$, the first query is the
prescribed $x_0$.  Each iteration evaluates one gradient, so ITEM-wf makes
exactly $N$ oracle calls.  Its coefficients depend only on the class
parameters, $N$, $w$, and $\eta$, so it belongs to $\mathcal A_{N,d}$ in every
dimension.

\subsection{Potential evolution}
\label{S:itemw-certificate}

This subsection establishes the one-step potential decrease underlying the ITEM-wf
guarantee.  We define interpolation slacks and a potential whose decrease at
each iteration is a nonnegative combination of two such slacks.  The endpoint
comparisons needed to turn these decreases into the final function-value bound
are given in the next subsection.

\begin{definition}[Interpolation slack]\label{D:itemw-delta}
Let $f\in\Fclass(\real^d)$ with minimizer $x_*$.  For $x\in\real^d$, write
\[
    x^+:=x-L^{-1}\nabla f(x).
\]
For $x,z\in\real^d$, the \emph{interpolation slack} is
\begin{equation}\label{E:itemw-delta}
    \Delta(x,z):=f(x)-f(z)-\ip{\nabla f(z)}{x-z}
       -\frac1{2L}\norm{\nabla f(x)-\nabla f(z)}^2
       -\frac{\mu}{2\bar q}\norm{x^+-z^+}^2.
\end{equation}
For the queried points of Algorithm~\ref{A:itemw} and the minimizer, write
\[
    \Delta_{i,j}:=\Delta(x_i,x_j),
    \qquad i,j\in\{0,\ldots,N-1,*\}.
\]
\end{definition}

Thus $\Delta(x,z)$ is the slack in the smooth $\mu$-strongly convex
interpolation inequality of \cite[Theorem~4]{TaylorHendrickxGlineur2017},
and in particular
\begin{equation}\label{E:itemw-delta-positive}
    \Delta(x,z)\geq0,
    \qquad x,z\in\real^d.
\end{equation}

The inner products in the potential argument are converted into slacks by
the following identity.

\begin{lemma}[Three-point identity]\label{L:three-point}
Let $f\in\Fclass(\real^d)$ with minimizer $x_*$.  For every
$x,z\in\real^d$,
\begin{equation}\label{E:itemw-delta-identity}
    \frac1{\bar q}
    \ip{x^+-x_*}{\nabla f(z)-\mu(z-x_*)}
      =\Delta(x,x_*)-\Delta(x,z)+\Delta(x_*,z).
\end{equation}
\end{lemma}

\begin{proof}
Expanding the three slacks by~\eqref{E:itemw-delta}, the function values
cancel, and the remaining terms are
\[
    \ip{\nabla f(z)}{x-x_*}-\frac1L\ip{\nabla f(x)}{\nabla f(z)}
    -\frac{\mu}{\bar q}
     \Bigl\langle x-x_*-\tfrac1L\nabla f(x),\,
                    z-x_*-\tfrac1L\nabla f(z)\Bigr\rangle,
\]
where the three $\mu$-terms were combined by
$\norm A^2-\norm{A-B}^2+\norm B^2=2\ip AB$.  Factoring out
$x^+-x_*=x-x_*-\frac1L\nabla f(x)$ and using
$1+q/\bar q=1/\bar q$ gives the left-hand side.
\end{proof}

For $1\leq k\leq N$, define the scaled states
\begin{equation}\label{D:itemw-scaled-states}
\begin{aligned}
    X_k&:=z_k-x_*-a_k(v_{k-1}-x_*), \\
    Y_k&:=b_k(v_{k-1}-x_*).
    \end{aligned}
\end{equation}
The update of $z_k$ gives the equivalent representation
\begin{equation}\label{E:itemw-output-state}
    X_k=(1-a_k)(x_{k-1}^+-x_*),
    \qquad 1\leq k\leq N.
\end{equation}
We associate with these states the potential
\begin{equation}\label{D:itemw-potential}
    V_k:=(1-a_k)\Delta_{k-1,*}
      +\frac{L}{2\bar q}\bigl(q\norm{X_k}^2+\norm{Y_k}^2\bigr),
    \qquad 1\leq k\leq N.
\end{equation}
The following lemma shows that the potential decreases along the interior
iterations of ITEM-wf.

\begin{lemma}[Potential decrease]\label{L:itemw-decrease}
Fix $L>0$, $0\leq\mu<L$, $N\geq1$, put $q=\mu/L$, and let $w$
satisfy the simplex conditions~\eqref{D:weight-simplex}.  Let $\eta>0$
satisfy $r_{\eta,w}\in\mathcal O_N(q)$, let $f\in\Fclass(\real^d)$ have
minimizer $x_*$, and let
the iterates be generated by Algorithm~\ref{A:itemw}.  For every
$1\leq k<N$,
\begin{equation}\label{E:itemw-decrease}
    V_k-V_{k+1}
      =(1-a_k)\Delta_{k-1,k}+(a_k-a_{k+1})\Delta_{*,k}\geq0.
\end{equation}
\end{lemma}

\begin{proof}
Fix $1\leq k<N$ and put
\[
    G_k:=x_k-x_k^+-q(x_k-x_*)
       =L^{-1}\bigl(\nabla f(x_k)-\mu(x_k-x_*)\bigr).
\]
Applying~\eqref{E:itemw-delta-identity} with
$(x,z)=(x_{k-1},x_k)$ and using~\eqref{E:itemw-output-state},
\begin{align*}
&(1-a_k)\Delta_{k-1,k}+(a_k-a_{k+1})\Delta_{*,k}\\
&\quad={}
 (1-a_k)\Delta_{k-1,*}+(1-a_{k+1})\Delta_{*,k}
 -\frac{L}{\bar q}\ip{X_k}{G_k}.
\end{align*}
The instance $x=z=x_k$ of~\eqref{E:itemw-delta-identity}, again using
\eqref{E:itemw-output-state}, gives
$L\ip{X_{k+1}}{G_k}/\bar q
=(1-a_{k+1})(\Delta_{k,*}+\Delta_{*,k})$.  Hence
\begin{align}
&(1-a_k)\Delta_{k-1,k}+(a_k-a_{k+1})\Delta_{*,k}\notag\\
&\quad={}
 (1-a_k)\Delta_{k-1,*}-(1-a_{k+1})\Delta_{k,*}
 +\frac{L}{\bar q}\ip{X_{k+1}-X_k}{G_k}.
\label{E:itemw-decrease-interpolation}
\end{align}

It remains to identify the last inner product with the decrease of the
quadratic part of the potential.  The definition of $x_k$ in
Algorithm~\ref{A:itemw} is equivalently
\[
    \bar q(1-a_{k+1})(x_k-x_*)
      =d_{k+1}(z_k-x_*)-a_{k+1}(v_{k-1}-x_*).
\]
Substituting~\eqref{D:itemw-scaled-states} and the first row
of~\eqref{E:recurrence-interface-inverse} gives
\[
    \bar q(1-a_{k+1})(x_k-x_*)=d_{k+1}X_k+c_{k+1}Y_k.
\]
Since $x_k^+-x_*=\bar q(x_k-x_*)-G_k$, multiplying by $1-a_{k+1}$ and
using~\eqref{E:itemw-output-state} at index $k+1$ yields
\[
    X_{k+1}
      =d_{k+1}X_k+c_{k+1}Y_k-(1-a_{k+1})G_k.
\]
Likewise, the definition of $h_k$ gives
$h_k=G_k+q(z_k-x_*)$.  Using the update
$v_k=v_{k-1}-c_{k+1}h_k/b_{k+1}$ from Algorithm~\ref{A:itemw} and the
second row of~\eqref{E:recurrence-interface-inverse} gives
\[
    Y_{k+1}
      =-qc_{k+1}X_k+d_{k+1}Y_k-c_{k+1}G_k.
\]
The ordered-sequence conditions and~\eqref{D:backward-coupling} give the
positive scalar
\[
    \tau_k:=\frac{c_{k+1}}{1+d_{k+1}}
      =\frac{1-a_{k+1}}{c_{k+1}}.
\]
Rearranging the last two identities,
\[
\begin{aligned}
    X_{k+1}-X_k&=\tau_k(Y_{k+1}+Y_k),\\
    Y_{k+1}-Y_k&=-\tau_k\bigl(q(X_{k+1}+X_k)+2G_k\bigr).
\end{aligned}
\]
Therefore
\[
\begin{aligned}
&q\norm{X_k}^2+\norm{Y_k}^2
  -q\norm{X_{k+1}}^2-\norm{Y_{k+1}}^2\\
&\quad=-q\ip{X_{k+1}+X_k}{X_{k+1}-X_k}
       -\ip{Y_{k+1}+Y_k}{Y_{k+1}-Y_k}\\
&\quad=-q\tau_k\ip{X_{k+1}+X_k}{Y_{k+1}+Y_k}
       +\tau_k\ip{Y_{k+1}+Y_k}
          {q(X_{k+1}+X_k)+2G_k}\\
&\quad=2\tau_k\ip{Y_{k+1}+Y_k}{G_k}
 =2\ip{X_{k+1}-X_k}{G_k}.
\end{aligned}
\]
Substituting this identity into~\eqref{E:itemw-decrease-interpolation} and
using~\eqref{D:itemw-potential} proves~\eqref{E:itemw-decrease}.  Its
right-hand side is nonnegative by~\eqref{E:sequence-order} and
\eqref{E:itemw-delta-positive}.
\end{proof}

\subsection{The matching upper bound}
\label{S:itemw-bound}

The potential decreases control the interior iterations, leaving two
endpoint comparisons.  At the initial point, $V_1$ is compared with the
weighted initial quantity; this is the only place where the weights appear
explicitly, and the discrepancy is controlled by the matching residual.  At
the final point, $V_N$ is compared with the function-value suboptimality at
$z_N$.  Combining these comparisons with the local decreases yields the
matching upper bound.

\begin{lemma}\label{L:itemw-boundaries}
Fix $L>0$, $0\leq\mu<L$, $N\geq1$, put $q=\mu/L$, and let $w$
satisfy the simplex conditions~\eqref{D:weight-simplex}.  Let $\eta>0$
satisfy $r_{\eta,w}\in\mathcal O_N(q)$, let $f\in\Fclass(\real^d)$ have
minimizer $x_*$, and let
the iterates be generated by Algorithm~\ref{A:itemw}.  Then the
potential~\eqref{D:itemw-potential} satisfies:
\begin{enumerate}[label=(\roman*),leftmargin=2.2em]
    \item At the initial point,
    \begin{equation}\label{E:itemw-initial-boundary}
    \begin{aligned}
    &\eta\biggl(
           w_x\frac L2\norm{x_0-x_*}^2
           +w_f(f(x_0)-f_*)
           +\frac{w_g}{2L}\norm{\nabla f(x_0)}^2
         \biggr)-V_1\\
    &\qquad=\left(\eta w_x+\sqrt{\frac{\eta}{\bar q}}\,b_1\right)\Delta_{*,0}
      -H_w(\eta)\,\Delta_{0,*}
      -\frac{qLH_w(\eta)}{\bar q}\norm{x_0^+-x_*}^2,
    \end{aligned}
    \end{equation}
    and the right-hand side is nonnegative if $H_w(\eta)\leq0$.
    \item At the output,
    \begin{equation}\label{E:itemw-terminal-boundary}
    \begin{aligned}
    V_N-(f(z_N)-f_*)={}&
     (1-a_N)\Delta(x_{N-1},z_N)
     +a_N\Delta(x_*,z_N)\\
    &+\frac{L}{2\bar q}
     \norm{(z_N^+-x_*)-X_N}^2
     \geq0.
    \end{aligned}
    \end{equation}
    The point $z_N^+$ and the values $f(z_N),\nabla f(z_N)$ are used only in
    the analysis and do not require an additional oracle call.
\end{enumerate}
\end{lemma}

\begin{proof}
\emph{Initial boundary.}
Since $z_0=x_0$, the first iteration of Algorithm~\ref{A:itemw} gives
$h_0=x_0-x_0^+=L^{-1}\nabla f(x_0)$.  Put
\[
    p_0^+:=x_0^+-x_*,
\]
so that $x_0-x_*=p_0^++h_0$, and denote the weighted initial quantity by
\[
    \mathcal E_0
    :=w_x\frac L2\norm{x_0-x_*}^2+w_f(f(x_0)-f_*)
      +\frac{w_g}{2L}\norm{\nabla f(x_0)}^2.
\]
From the initialization in Algorithm~\ref{A:itemw}, the definition of
$c_1$ in~\eqref{D:recurrence-interface-coefficients}, and
\eqref{E:itemw-output-state},
\[
    X_1=(1-a_1)p_0^+,
    \qquad
    Y_1=b_1p_0^+-\sqrt{\bar q\eta}\,h_0.
\]
Substituting these expressions into the potential and then expanding the two
squared norms gives
\[
\begin{aligned}
\eta\mathcal E_0-V_1
&=\eta w_f(f(x_0)-f_*)-(1-a_1)\Delta_{0,*}\\
&\quad
 +\frac{L}{2\bar q}\Bigl(
      \eta\bar q w_x\norm{p_0^++h_0}^2
      +\eta\bar q w_g\norm{h_0}^2
      -q(1-a_1)^2\norm{p_0^+}^2
      -\norm{b_1p_0^+-\sqrt{\bar q\eta}\,h_0}^2
   \Bigr)\\
&=\eta w_f(f(x_0)-f_*)-(1-a_1)\Delta_{0,*}\\
&\quad
 +\frac{L}{2\bar q}\Bigl(
      \bigl[\eta\bar q w_x-q(1-a_1)^2-b_1^2\bigr]\norm{p_0^+}^2\\
&\hspace{34mm}
      +2\bigl[\eta\bar q w_x+b_1\sqrt{\bar q\eta}\bigr]
          \ip{p_0^+}{h_0}
      -\eta\bar q w_f\norm{h_0}^2
   \Bigr),
\end{aligned}
\]
where the second equality also uses $w_x+w_f+w_g=1$.
The invariant~\eqref{E:recurrence-interface-factorizations}, the radius
relation~\eqref{D:risk-coefficient-induced-radius}, and the matching
residual~\eqref{E:mixed-matching-residual} give
\[
\begin{aligned}
    b_1^2&=q(1-a_1^2)+\eta(\bar q w_x-qw_f),\\
    1-a_1-\eta w_f-H_w(\eta)
      &=\eta w_x+\sqrt{\frac{\eta}{\bar q}}\,b_1.
\end{aligned}
\]
Substituting these identities and collecting the terms proportional to
$w_f$ continues the calculation as
\[
\begin{aligned}
\eta\mathcal E_0-V_1
&=\eta w_f(f(x_0)-f_*)-(1-a_1)\Delta_{0,*}\\
&\quad
 +\frac{L}{2\bar q}\Bigl(
      q\bigl(\eta w_f-2(1-a_1)\bigr)\norm{p_0^+}^2\\
&\hspace{31mm}
      +2\bar q\bigl(1-a_1-\eta w_f-H_w(\eta)\bigr)
          \ip{p_0^+}{h_0}
      -\eta\bar q w_f\norm{h_0}^2
   \Bigr)\\
&=-(1-a_1)\Delta_{0,*}
  -\frac{qL(1-a_1)}{\bar q}\norm{p_0^+}^2
  +L\bigl(1-a_1-H_w(\eta)\bigr)\ip{p_0^+}{h_0}\\
&\quad
 +\eta w_f\left(
      f(x_0)-f_*-L\ip{p_0^+}{h_0}
      -\frac L2\norm{h_0}^2
      +\frac{qL}{2\bar q}\norm{p_0^+}^2
   \right)\\
&=-\eta w_f\Delta_{*,0}-(1-a_1)\Delta_{0,*}
  -\frac{qL(1-a_1)}{\bar q}\norm{p_0^+}^2
  +L\bigl(1-a_1-H_w(\eta)\bigr)\ip{p_0^+}{h_0}.
\end{aligned}
\]
The last equality uses~\eqref{E:itemw-delta} at $(x_*,x_0)$, for which the
parenthesized expression is $-\Delta_{*,0}$.  The remaining inner product is
converted to interpolation slacks by the three-point
identity~\eqref{E:itemw-delta-identity} at $x=z=x_0$,
\[
    L\ip{p_0^+}{h_0}
      =\Delta_{0,*}+\Delta_{*,0}
        +\frac{qL}{\bar q}\norm{p_0^+}^2.
\]
Substituting this identity gives
\[
\begin{aligned}
\eta\mathcal E_0-V_1
&=\bigl(1-a_1-\eta w_f-H_w(\eta)\bigr)\Delta_{*,0}
  -H_w(\eta)\left(
      \Delta_{0,*}+\frac{qL}{\bar q}\norm{p_0^+}^2
   \right)\\
&=\left(\eta w_x+\sqrt{\frac{\eta}{\bar q}}\,b_1\right)\Delta_{*,0}
 -H_w(\eta)\Delta_{0,*}
 -\frac{qLH_w(\eta)}{\bar q}\norm{x_0^+-x_*}^2,
\end{aligned}
\]
where the last equality uses the second identity above and the definition of
$p_0^+$.
This proves~\eqref{E:itemw-initial-boundary}.

If $H_w(\eta)\leq0$, then $b_1>0$ by ordering and all interpolation
slacks are nonnegative by~\eqref{E:itemw-delta-positive}; hence every term
on the right-hand side of~\eqref{E:itemw-initial-boundary} is nonnegative.

\emph{Terminal boundary.}
We rewrite $V_N$ in terms of $z_N$ and $X_N$.  The terminal interpolation
slack then puts the function-value term in the same quadratic variables,
after which the three-point identity removes the remaining inner product.
From the definition of $z_N$ and~\eqref{D:itemw-scaled-states},
\[
    z_N-x_*=X_N+a_N(v_{N-1}-x_*).
\]
Since $b_N=\sqrt{\bar q}\,a_N$, this is equivalent to
\begin{equation}\label{E:itemw-terminal-state}
    Y_N=\sqrt{\bar q}\bigl[(z_N-x_*)-X_N\bigr].
\end{equation}
Write
\[
    G:=z_N-z_N^+-q(z_N-x_*)
      =L^{-1}\bigl(\nabla f(z_N)-\mu(z_N-x_*)\bigr).
\]

Using~\eqref{E:itemw-terminal-state} in the definition of $V_N$ first gives
\[
\begin{aligned}
V_N-(f(z_N)-f_*)
&=(1-a_N)\Delta_{N-1,*}-(f(z_N)-f_*)\\
&\quad+\frac{L}{2\bar q}\left(
     q\norm{X_N}^2
     +\bar q\norm{(z_N-x_*)-X_N}^2
   \right).
\end{aligned}
\]
For $(x,z)=(x_*,z_N)$, the interpolation-slack definition specializes to
\[
    \Delta(x_*,z_N)
      =-(f(z_N)-f_*)
       +\frac{L}{2\bar q}\left(
           \bar q\norm{z_N-x_*}^2-\norm{z_N^+-x_*}^2
         \right).
\]
Using this expression and then
$z_N^+-x_*=\bar q(z_N-x_*)-G$ gives the chained calculation
\[
\begin{aligned}
V_N-(f(z_N)-f_*)
&=(1-a_N)\Delta_{N-1,*}+\Delta(x_*,z_N)\\
&\quad+\frac{L}{2\bar q}\Bigl(
     q\norm{X_N}^2
     +\bar q\norm{(z_N-x_*)-X_N}^2\\
&\hspace{35mm}
     -\bar q\norm{z_N-x_*}^2+\norm{z_N^+-x_*}^2
   \Bigr)\\
&=(1-a_N)\Delta_{N-1,*}+\Delta(x_*,z_N)
  -\frac{L}{\bar q}\ip{X_N}{G}\\
&\quad+\frac{L}{2\bar q}
   \norm{(z_N^+-x_*)-X_N}^2.
\end{aligned}
\]

Finally, applying~\eqref{E:itemw-delta-identity} at
$(x,z)=(x_{N-1},z_N)$ and using
$X_N=(1-a_N)(x_{N-1}^+-x_*)$ identifies the inner product as
\[
\frac{L}{\bar q}\ip{X_N}{G}
=(1-a_N)\bigl[
    \Delta_{N-1,*}-\Delta(x_{N-1},z_N)+\Delta(x_*,z_N)
\bigr].
\]
Substitution leaves
\[
\begin{aligned}
V_N-(f(z_N)-f_*)
&=(1-a_N)\Delta(x_{N-1},z_N)
  +a_N\Delta(x_*,z_N)\\
&\quad+\frac{L}{2\bar q}
   \norm{(z_N^+-x_*)-X_N}^2.
\end{aligned}
\]
This proves~\eqref{E:itemw-terminal-boundary}.  Since the associated sequence
is ordered, $0<a_N<1$, and the interpolation slacks are nonnegative by
\eqref{E:itemw-delta-positive}; hence every term on the right-hand side is
nonnegative.
\end{proof}

\begin{theorem}[ITEM-wf upper bound]
\label{T:itemw-performance}
Fix $L>0$, $0\leq\mu<L$, $N\geq1$, put $q=\mu/L$, and let $w$
satisfy the simplex conditions~\eqref{D:weight-simplex}.  Let $\eta>0$
satisfy $r_{\eta,w}\in\mathcal O_N(q)$ and $H_w(\eta)\leq0$.  For every
dimension $d\geq1$, every
$f\in\Fclass(\real^d)$ with a minimizer $x_*$, and every initial point
$x_0\in\real^d$, ITEM-wf (Algorithm~\ref{A:itemw}) constructed from $\eta$
produces an output $z_N$ satisfying
\begin{equation}\label{E:itemw-upper}
    f(z_N)-f_*\leq\eta\biggl(\frac{Lw_x}{2}\norm{x_0-x_*}^2+w_f(f(x_0)-f_*)+\frac{w_g}{2L}\norm{\nabla f(x_0)}^2\biggr).
\end{equation}
\end{theorem}

\begin{proof}
If $H_w(\eta)\leq0$, the initial comparison~\eqref{E:itemw-initial-boundary}
gives
\[
    V_1\leq\eta\left(
       \frac{Lw_x}{2}\norm{x_0-x_*}^2
       +w_f(f(x_0)-f_*)
       +\frac{w_g}{2L}\norm{\nabla f(x_0)}^2
    \right).
\]
By Lemma~\ref{L:itemw-decrease}, $V_1\geq\cdots\geq V_N$, while the
terminal comparison~\eqref{E:itemw-terminal-boundary} gives
$f(z_N)-f_*\leq V_N$.  Combining these inequalities yields
\eqref{E:itemw-upper}.
\end{proof}

\begin{remark}
A closely related variant of ITEM-wf has an interesting uniformity property.  In ITEM-wf, 
all iterations after the first depend on the initial-weight vector through induced radius $r_{\eta,w}$ and are therefore independent of the particular choice of $\eta$ and $w$. The first iteration is different, as the definition of $c_1$ depends additionally on $\eta$.  If the first iteration is replaced by an ordinary gradient step with step size $L^{-1}$, the resulting method will therefore be identical for all weights $w$ whose exact coefficients induce the same radius.
We conjecture that this common method admits a worst-case guarantee that is within a factor four of the exact minimax value and that this bound holds uniformly over any choice of weights within that family. 
We leave the proof of this conjecture for future work.
\end{remark}

\section{Exact minimax characterization}
\label{S:weighted-exact}

The lower bound of Theorem~\ref{T:mixed-realization} and the upper
bound of Theorem~\ref{T:itemw-performance} are built from the same scalar
$\eta$.  The lower bound holds when $H_w(\eta)\geq0$, because the residual
increases the realized certificate in~\eqref{E:mixed-realized-lower}; the
upper bound holds when $H_w(\eta)\leq0$, because the initial comparison
bounds $V_1$ by the weighted initial quantity.  At a root, both bounds hold,
which proves the main theorem and determines the risk coefficient uniquely.

\begin{proof}[Proof of Theorem~\ref{T:intro-main}]
A risk coefficient exists by Proposition~\ref{P:mixed-existence}.  For
any risk coefficient $\eta$, $R>0$, and $d\geq2N+1$,
Theorem~\ref{T:mixed-realization} and Theorem~\ref{T:itemw-performance},
the latter applied to ITEM-wf constructed from $\eta$, give
\[
    \frac{L}{2}\,\eta R^2
    \leq \risk{N}{\mathcal P^{\mu,L}_{w,R}(\real^d)}
    \leq \frac{L}{2}\,\eta R^2.
\]
Hence the minimax risk equals $\frac{L}{2}\,\eta R^2$ for every
risk coefficient $\eta$; since $R>0$, the risk coefficient is
therefore unique, and denoting it by $\eta_N^\star(q,w)$
gives~\eqref{E:weighted-exact-minimax}.  The upper bound uses no dimension
assumption, so ITEM-wf satisfies it in every dimension.
\end{proof}

Two specializations make the scalar characterization explicit.  On the
convex boundary $q=0$ with $w_x>0$, the matching equation admits a closed-form
solution; at the distance vertex $w=w^x$, this solution recovers the exact OGM
rate.  For $q\in(0,1)$, the balanced relation
\[
    \bar q w_x=qw_f
\]
makes the induced radius~\eqref{D:risk-coefficient-induced-radius} independent
of $\eta$, with value $\sqrt q$.  Thus all weights satisfying this relation use
the same ordered backward sequence, and only the initial matching equation
depends on the particular weights.  Proposition~\ref{P:balanced-risk-coefficient}
solves this matching equation explicitly and gives the corresponding exact
minimax rate.  Both specializations are derived in
Appendix~\ref{A:specializations}.

\begin{remark}[Computation]
The explicit specializations described above require no root search.  In the
remaining cases with $q>0$, evaluating $H_w(\eta)$, whenever
$r_{\eta,w}\in\mathcal O_N(q)$, requires one backward pass through
Definition~\ref{D:backward-sequence}, and hence $O(N)$ scalar operations and
constant memory.  Let $J_w=[0,\eta_*)$ be the ordered component containing
zero from the proof of Proposition~\ref{P:mixed-existence}.  Theorem~\ref{T:mixed-realization},
Theorem~\ref{T:itemw-performance}, and the exact characterization above imply
that $H_w(\eta)>0$ for $0\leq\eta<\eta_N^\star(q,w)$ and
$H_w(\eta)<0$ for $\eta_N^\star(q,w)<\eta<\eta_*$.  The existence proof
provides a point $\eta_+\in J_w$ with $H_w(\eta_+)<0$; hence
$[0,\eta_+]$ brackets the unique root and lies entirely in the ordered
region.  Bisection on this interval therefore computes
$\eta_N^\star(q,w)$ to absolute accuracy $0<\varepsilon<\eta_+$ using
$O\!\left(N\log(\eta_+/\varepsilon)\right)$ scalar operations.
\end{remark}

\section{Concluding remarks}\label{S:concluding}

This paper gives an exact finite-horizon minimax characterization of
function-value suboptimality under nonnegative weighted combinations of the
initial squared distance, function-value suboptimality, and squared gradient
norm.  Whenever
$q>0$ or $w_x>0$, the exact deterministic minimax risk after $N$ first-order
oracle calls is
$\frac{L}{2}\,\eta_N^\star(q,w)R^2$ in dimension $d\geq2N+1$, where
$\eta_N^\star(q,w)$ is the unique risk coefficient characterized by the
matching equation~\eqref{E:mixed-matching-equation}; ITEM-wf attains the same upper
bound in every dimension.  A single backward recurrence underlies both sides
of this characterization, determining the admissible hard chain and the
coefficients of ITEM-wf.  The hard chain is optimal within the family of
admissible chains: by Corollary~\ref{C:certificate-bound}, no other member
yields a larger certificate value relative to its weighted budget at the risk
coefficient.

For the initial function-value condition,
Proposition~\ref{P:gap-correspondence} identifies the exact risk coefficient with the
ITEM-f contraction factor $\Upsilon_N^{-2}$.  The existing ITEM-f upper
guarantee is therefore matched here by an oracle-complexity lower bound,
establishing its exact finite-horizon minimax optimality.  On the distance
axis, the convex limit recovers the exact OGM value
of~\cite{Drori2017Exact,KimFessler2016}.  On the balanced slice
$\bar q w_x=qw_f$, Appendix~\ref{A:sqrtq-specialization} gives the exact
coefficient explicitly; the gradient axis is the endpoint $u=0$ of this slice.

Gradient-norm minimization gives a complementary finite-horizon problem.  In
the smooth convex setting and for $d\geq N+2$, Grimmer, Jo, and Park recently
proved in~\cite{GrimmerJoPark2026} that OGM-G is the unique fixed-step first-order method minimizing the
worst-case final squared gradient norm under a bound on the initial
function-value suboptimality, while the Lemniscate method~\cite{KimRyuDasGupta2026} is uniquely optimal
in the same class under a bound on the initial squared distance to a
minimizer.  Here fixed-step means that the method's
coefficients are prescribed in advance and do not depend on the observed
function values or gradients; optimality of OGM-G and Lemniscate among general
deterministic first-order methods remains open.  The zero-chain construction
developed in this paper is tailored to terminal function-value suboptimality
and, in its present form, does not yield matching lower bounds for either
OGM-G or Lemniscate.  Extending the hard-instance construction to terminal
squared gradient norm is therefore a natural direction toward corresponding
information-theoretic optimality results.

\section*{Declaration of generative AI and AI-assisted technologies}
During the preparation of this work the author extensively used publicly
available generative AI and AI-assisted technologies for exploratory
derivations, algebraic verification, drafting, and revision.  The author
reviewed and checked the mathematical statements and proofs and takes full
responsibility for the content of the publication.

\appendix

\section{Deferred proofs}
\label{A:deferred-proofs}

This appendix collects the proofs deferred from the main text, together
with the auxiliary constructions they use.

\subsection{Proposition~\ref{P:revealed-face}: projection onto the revealed convex hull}
\label{A:revealed-face-proof}

To prove Proposition~\ref{P:revealed-face}, we show that at each stage $j$,
weight on every future chain site $s_m$, $m>j$, can be transferred to $s_j$,
while weight on $s_*$ can be transferred to a suitable replacement point in
$\conv\{s_0,\ldots,s_j\}$.  The next lemma supplies the inequalities that
certify these transfers.

\begin{lemma}[Weight-transfer certificates]\label{L:generic-elimination}
Let $\mathcal C$ be an admissible chain, and let $(s_i)_{i\in I}$ be its
sites from Definition~\ref{D:chain-function}.  Then, for every $0\leq j<N$:
\begin{enumerate}[label=(\roman*),leftmargin=2.2em]
    \item $\ip{s_{j+1}-s_j}{s_i}\geq0$ for every $i\in I$;
    \item there exists $\widehat s_j\in\conv\{s_0,\ldots,s_j\}$, with $\widehat s_j\neq s_*$,
    such that
    \[
        s_*-\widehat s_j\perp S_{j-1},
        \qquad
        \ip{s_*-\widehat s_j}{s_i}\geq0,
        \quad i\in I.
    \]
\end{enumerate}
\end{lemma}

\begin{proof}
Admissibility gives $\gamma_k-q\delta_k>0$ and $D_k>\gamma_k$ for
$0\leq k<N$.  For $0\leq k<N-1$, the recurrence
~\eqref{E:generic-pi-recurrence} reads
\begin{equation}\label{E:pi-coefficient-recurrence}
    \pi_k\gamma_k-qD_k
      =(\gamma_k-q\delta_k)\pi_{k+1}>0.
\end{equation}
At the terminal coordinate, when $q>0$, terminal admissibility gives
\[
    \bigl(\pi_{N-1}\gamma_{N-1}-qD_{N-1}\bigr)
    (D_{N-1}-\gamma_{N-1})\geq qD_N^2>0.
\]
Since $D_{N-1}-\gamma_{N-1}>0$, the first factor is positive.  For $q=0$,
the same conclusion is immediate from
$\pi_{N-1}\gamma_{N-1}>0$.  Together with
\eqref{E:pi-coefficient-recurrence} and $\pi_k\leq1$ for $k<N$, we therefore have
\[
    qD_k<\pi_k\gamma_k\leq\gamma_k,
    \qquad 0\leq k<N.
\]

The site offsets of Definition~\ref{D:chain-function} now give
a common bound for all sites, including the minimizer site.  Every coordinate
of every $t_i$, $i\in I$, is nonnegative; moreover, for $0\leq k<N$,
$(t_i)_k\leq\gamma_k$.  For the chain sites this follows from
$q\delta_k<\gamma_k$, while for $t_*$ it follows from the preceding coordinate bound.  Since
$\bar q s_i=\widetilde x_*+t_i$, we therefore have
\begin{equation}\label{E:rescaled-site-coordinates}
    (\bar q s_i)_\ell\geq-D_\ell,
    \qquad 0\leq\ell\leq N,
    \qquad
    (\bar q s_i)_k\leq-D_k+\gamma_k,
    \qquad 0\leq k<N,
\end{equation}
for every $i\in I$.

\emph{Part (i): adjacent-site transfer.}
Fix $0\leq k<N$.  By~\eqref{D:generic-transformed-sites},
\[
    \bar q(s_{k+1}-s_k)=t_{k+1}-t_k
      =-(\gamma_k-q\delta_k)e_k+\gamma_{k+1}e_{k+1}.
\]
Hence, for every $i\in I$, the coordinate bounds above give
\[
\begin{aligned}
 \ip{t_{k+1}-t_k}{\bar q s_i}
   &=-(\gamma_k-q\delta_k)(\bar q s_i)_k
      +\gamma_{k+1}(\bar q s_i)_{k+1}\\
   &\geq (\gamma_k-q\delta_k)(D_k-\gamma_k)
      -\gamma_{k+1}D_{k+1}\geq0,
\end{aligned}
\]
where the last inequality is~\eqref{E:generic-balance}.  This proves
part~(i).

\emph{Part (ii): minimizer replacement.}
Fix $0\leq j<N$ and define
\[
    \widehat s_j
    :=\sum_{\ell=0}^{j-1}(\pi_\ell-\pi_{\ell+1})s_\ell+\pi_js_j.
\]
The coefficients are nonnegative by~\eqref{E:generic-pi-monotone} and
telescope to one, so $\widehat s_j\in\conv\{s_0,\ldots,s_j\}$.  Using $\bar q s_\ell=\widetilde x_*+t_\ell$ and the fact that the
coefficients defining $\widehat s_j$ sum to one,
\[
\bar q\widehat s_j-\widetilde x_*
 =\sum_{\ell=0}^{j-1}(\pi_\ell-\pi_{\ell+1})t_\ell+\pi_jt_j.
\]
For $\ell<j$, the $e_\ell$-coordinate of the right-hand side is
\[
\begin{aligned}
&(\pi_\ell-\pi_{\ell+1})\gamma_\ell
 +q\delta_\ell\left[
    \sum_{r=\ell+1}^{j-1}(\pi_r-\pi_{r+1})+\pi_j
  \right]\\
&\qquad={}
 (\pi_\ell-\pi_{\ell+1})\gamma_\ell
   +\pi_{\ell+1}q\delta_\ell\\
&\qquad={}
 \pi_\ell\gamma_\ell
   -\pi_{\ell+1}(\gamma_\ell-q\delta_\ell)
 =qD_\ell,
\end{aligned}
\]
where the bracket telescopes to $\pi_{\ell+1}$ and the last equality is
\eqref{E:pi-coefficient-recurrence}; its $e_j$-coordinate is
$\pi_j\gamma_j$, and all higher coordinates vanish.  Hence
\[
\bar q\widehat s_j-\widetilde x_*
  =q\sum_{\ell=0}^{j-1}D_\ell e_\ell+\pi_j\gamma_je_j,
\]
and therefore
\[
    \bar q(s_*-\widehat s_j)
      =-(\pi_j\gamma_j-qD_j)e_j
        +q\sum_{\ell=j+1}^ND_\ell e_\ell.
\]
This vector is orthogonal to $S_{j-1}$ and is nonzero: if $j<N-1$,
\eqref{E:pi-coefficient-recurrence} gives
$\pi_j\gamma_j-qD_j>0$, while for $j=N-1$ the same positivity follows from
the terminal calculation above.  At the terminal
stage, this identity and~\eqref{E:rescaled-site-coordinates}
give, for every $i\in I$,
\[
\begin{aligned}
 \ip{\bar q(s_*-\widehat s_{N-1})}{\bar q s_i}
   &=-\bigl(\pi_{N-1}\gamma_{N-1}-qD_{N-1}\bigr)
      (\bar q s_i)_{N-1}+qD_N(\bar q s_i)_N\\
   &\geq
      \bigl(\pi_{N-1}\gamma_{N-1}-qD_{N-1}\bigr)
      (D_{N-1}-\gamma_{N-1})-qD_N^2\geq0,
\end{aligned}
\]
where the last inequality is~\eqref{E:generic-terminal-admissibility}.
Finally, for $0\leq k<N-1$,
\[
    \bar q(\widehat s_{k+1}-\widehat s_k)
      =\pi_{k+1}(t_{k+1}-t_k).
\]
Summing from $k=j$ to $N-2$ gives
\[
    \bar q(s_*-\widehat s_j)
      =\bar q(s_*-\widehat s_{N-1})
        +\sum_{k=j}^{N-2}\pi_{k+1}(t_{k+1}-t_k).
\]
Taking inner products with $\bar q s_i$, the first term is nonnegative by
the terminal calculation and every summand is nonnegative by part~(i) and
$\pi_{k+1}>0$.  Hence
$\ip{s_*-\widehat s_j}{s_i}\geq0$ for every $i\in I$, proving part~(ii).
\end{proof}

\begin{proof}[Proof of Proposition~\ref{P:revealed-face}]
Write $p:=\Proj_{K_{\mathcal C}}(x)$.  Since $p\in K_{\mathcal C}$,
fix a convex representation $p=\sum_{i\in I}\alpha_i s_i$.  First fix $m>j$.
By~\eqref{E:generic-site-expansion}, each difference $s_{k+1}-s_k$ is supported on
coordinates $k,k+1$, so
\[
    s_m-s_j=\sum_{k=j}^{m-1}(s_{k+1}-s_k)\perp S_{j-1}.
\]
Lemma~\ref{L:generic-elimination}(i), summed over $k$, gives
$\ip{s_m-s_j}{s_i}\geq0$ for every $i\in I$, hence
$\ip{s_m-s_j}{s_i-x}\geq0$.  Moreover $s_m\neq s_j$, because
$\bar q s_m-\widetilde x_*$ has a positive $e_m$-coordinate whereas
$\bar q s_j-\widetilde x_*$ does not.  Applying
Lemma~\ref{L:face-exchange} with $z=s_j$ gives $\alpha_m=0$.

For the minimizer site, let $\widehat s_j$ be the point supplied by
Lemma~\ref{L:generic-elimination}(ii).  Then $\widehat s_j\in K_{\mathcal C}$,
$\widehat s_j\neq s_*$, $s_*-\widehat s_j\perp S_{j-1}$, and
$\ip{s_*-\widehat s_j}{s_i}\geq0$ for every $i\in I$; since $x\in S_{j-1}$, the last
two give $\ip{s_*-\widehat s_j}{s_i-x}\geq0$.  A second application of
Lemma~\ref{L:face-exchange} therefore gives $\alpha_*=0$.
Thus $p\in\conv\{s_0,\ldots,s_j\}$, proving the first claim
in~\eqref{E:revealed-face-projection}.

By~\eqref{E:generic-site-expansion},
$\bar q s_i-\widetilde x_*=q\sum_{\ell=0}^{i-1}\delta_\ell e_\ell+\gamma_ie_i\in S_i$.
The gradient formula~\eqref{E:projection-gradient} therefore gives
\[
    \frac1L\nabla F_{\mathcal C}(x)
    =qx+\sum_{i=0}^j\alpha_i(\bar q s_i-\widetilde x_*)\in S_j,
\]
which proves the second claim.
\end{proof}

\subsection{Lemma~\ref{L:certificate-data}: certificate and initial-condition data}
\label{A:certificate-data-proof}

\begin{proof}[Proof of Lemma~\ref{L:certificate-data}]
We show that the certificate value $\Phi$ of an admissible chain bounds the
suboptimality of $F_{\mathcal C}$ on the final revealed subspace $S_{N-1}$ from
below, and that $B_x$, $B_f$, and $B_g$ are the squared initial distance, the scaled
initial suboptimality, and the scaled initial squared gradient norm of
$F_{\mathcal C}$ at $x_0=0$.

\emph{Part (i).}
Since $s_*=\widetilde x_*$ by~\eqref{D:generic-sites}, Lemma~\ref{L:smoothed-interpolant}(ii) gives
$F_{\mathcal C}(\widetilde x_*)=0$ and shows that $\widetilde x_*$ is a minimizer.
The support inequality~\eqref{E:support-ineq} gives
$F_{\mathcal C}(x)\geq Q_N(x)$ for every $x$.  The chain point
$\widetilde x_N=-\sum_{\ell=0}^{N-1}\delta_\ell e_\ell$ belongs to $S_{N-1}$, and
\[
    \nabla Q_N(\widetilde x_N)=\widetilde g_N=L\gamma_Ne_N\perp S_{N-1}.
\]
Because $Q_N$ is convex, $\widetilde x_N$ minimizes $Q_N$ over $S_{N-1}$.  Therefore,
for every $x\in S_{N-1}$,
\[
    F_{\mathcal C}(x)
    \geq Q_N(x)
    \geq Q_N(\widetilde x_N).
\]
Taking the infimum over $S_{N-1}$ proves the inequality in
\eqref{E:generic-restricted-value}.  To identify its value, complete the
square in $Q_N(\widetilde x_N)$:
\[
\begin{aligned}
 Q_N(\widetilde x_N)
 &=\frac{L}{2\bar q}\left[
    \bar qD_N^2-(D_N-\gamma_N)^2
    -q\sum_{\ell=0}^{N-1}(D_\ell-\delta_\ell)^2
  \right]\\
 &=\frac{L}{2}\Phi,
\end{aligned}
\]
which proves the equality in~\eqref{E:generic-restricted-value}.

\emph{Part (ii).}
Apply Proposition~\ref{P:revealed-face} with $j=0$ and $x=0$.  Since the revealed
convex hull then consists only of $s_0$,
\[
    \Proj_{K_{\mathcal C}}(0)=s_0.
\]
We identify the squared-distance, function-value suboptimality, and squared-gradient-norm
quantities separately.

First, by~\eqref{D:generic-vertices},
\[
    \norm{\widetilde x_*}^2=\sum_{\ell=0}^ND_\ell^2=B_x.
\]
For the function-value suboptimality, the squared-distance representation gives
\[
    F_{\mathcal C}(0)
      =\frac L2\norm{\widetilde x_*}^2
       -\frac{L-\mu}{2}\norm{s_0}^2.
\]
The vector $\bar q s_0=\widetilde x_*+t_0$ has coordinate
$-(D_0-\gamma_0)$ at $\ell=0$ and coordinate $-D_\ell$ for $\ell\geq1$.  Hence
\[
    \bar q^2\norm{s_0}^2
      =(D_0-\gamma_0)^2+\sum_{\ell=1}^ND_\ell^2.
\]
Substitution and $L-\mu=L\bar q$ give
\[
\begin{aligned}
    \frac{2}{L}F_{\mathcal C}(0)
      &=\sum_{\ell=0}^ND_\ell^2
        -\frac{(D_0-\gamma_0)^2+\sum_{\ell=1}^ND_\ell^2}{\bar q}\\
      &=D_0^2-
        \frac{(D_0-\gamma_0)^2+q\sum_{\ell=1}^ND_\ell^2}{\bar q}
       =B_f.
\end{aligned}
\]
For the gradient, the projection formula~\eqref{E:projection-gradient} at
$x=0$ gives
\[
\begin{aligned}
    \frac1L\nabla F_{\mathcal C}(0)
      &=-q\widetilde x_*+\bar q(s_0-\widetilde x_*)\\
      &=\bar q s_0-\widetilde x_*=t_0=\gamma_0e_0.
\end{aligned}
\]
Thus
\[
    \frac1{L^2}\norm{\nabla F_{\mathcal C}(0)}^2
      =\gamma_0^2=B_g.
\]
This proves the three identifications in~\eqref{E:generic-initial-quantities}.

The sign conditions for an admissible chain give $D_0>0$ and $\gamma_0>0$,
so $B_x,B_g>0$.  Also $\nabla F_{\mathcal C}(0)\neq0$, whereas $\widetilde x_*$ is a
minimizer.  Hence $0$ is not a minimizer, and therefore
$F_{\mathcal C}(0)>0$.  Thus $B_f>0$ as well.
\end{proof}

\subsection{Boundary identities}
\label{A:boundary-identities}

This subsection proves the two boundary lemmas of
Section~\ref{S:recurrence-matched}.

\begin{proof}[Proof of Lemma~\ref{L:recurrence-interface-matched}]
The definition~\eqref{E:mixed-matching-residual} of the residual gives
\[
    \sqrt{\bar q\eta}\,b_1
      =\bar q\bigl(1-a_1-\eta(w_x+w_f)-H_w(\eta)\bigr).
\]
Using the radius definition~\eqref{D:risk-coefficient-induced-radius} and the
invariant $b_1^2=r_{\eta,w}^2-qa_1^2$ from
\eqref{E:recurrence-interface-factorizations},
\[
\begin{aligned}
    c_1b_1
      &=b_1^2+\sqrt{\bar q\eta}\,b_1\\
      &=(1-a_1)(1+qa_1)-\eta w_f-\bar q H_w(\eta).
\end{aligned}
\]
The first row of~\eqref{D:recurrence-boundary} therefore gives
\[
    a_0
      =d_1a_1+c_1b_1
      =1-\eta w_f-\bar q H_w(\eta).
\]
The second row gives directly
\[
    b_0
      =-qc_1a_1+d_1b_1
      =\bar q\,b_1-qa_1\sqrt{\bar q\eta}.
\]

Finally, since $(c_1-b_1)^2=\bar q\eta$,
\[
\begin{aligned}
    c_1^2
      &=2c_1b_1-b_1^2+\bar q\eta\\
      &=(1-a_1)(1+\bar q+qa_1)
         +\eta(\bar q w_g-w_f)-2\bar q H_w(\eta),
\end{aligned}
\]
where the last line uses
\eqref{D:risk-coefficient-induced-radius} and $w_x+w_f+w_g=1$.
\end{proof}

\begin{proof}[Proof of Lemma~\ref{L:recurrence-boundary-identities}]
For $k=0$, the boundary relation~\eqref{D:recurrence-boundary} gives
\[
    a_k=d_{k+1}a_{k+1}+c_{k+1}b_{k+1},
    \qquad
    b_k=-qc_{k+1}a_{k+1}+d_{k+1}b_{k+1};
\]
for $1\leq k<N$, the same two identities are~\eqref{D:sequence} with
index $k+1$.  Hence, for every $0\leq k<N$,
\[
\begin{aligned}
    a_kb_{k+1}-a_{k+1}b_k
      &=c_{k+1}(qa_{k+1}^2+b_{k+1}^2)\\
      &=r_{\eta,w}^2c_{k+1},
\end{aligned}
\]
where the last equality is the invariant
\eqref{E:recurrence-interface-factorizations}.  This proves
\eqref{E:recurrence-boundary-cross}.  For $1\leq k<N$, rearranging the first
row of~\eqref{D:sequence} gives
\[
\begin{aligned}
    c_{k+1}b_{k+1}
      &=a_k-d_{k+1}a_{k+1}\\
      &=(1-a_{k+1})(1+qa_{k+1})-(1-a_k),
\end{aligned}
\]
which proves~\eqref{E:recurrence-boundary-product} at the interior indices.
At $k=0$, the same calculation follows from the first row of
\eqref{D:recurrence-boundary}.
\end{proof}

\subsection{Lemma~\ref{L:certificate-decomposition}: certificate decomposition}
\label{A:certificate-decomposition-proof}

The decomposition separates into a local quadratic calculation and a
telescoping cross-index correction.  We first record the identity supplied by
the recurrence.  The proof of Lemma~\ref{L:certificate-decomposition} then
specializes this identity at the initial and interior indices, treats the
terminal term directly, and combines the remaining cross-index terms into the
chain-condition slacks.

\begin{lemma}[Local quadratic identity]
\label{L:certificate-local-identity}
Let $N\geq1$, $q\in[0,1)$, let $w$ satisfy the simplex
conditions~\eqref{D:weight-simplex}, and let $\eta>0$ satisfy
$r_{\eta,w}\in\mathcal O_N(q)$.
Use the associated recurrence and boundary data from
Sections~\ref{S:recurrence-interface} and~\ref{S:recurrence-matched}, and put
$r:=r_{\eta,w}$.  Then, for every $0\leq k<N$ and every
$D,\gamma,\delta\in\mathbb R$,
\begin{align}
&(b_{k+1}D-c_{k+1}\gamma)^2
 +q\bigl(\delta-a_{k+1}D-(1-a_{k+1})\gamma\bigr)^2\notag\\
&\qquad
 +2(1-a_{k+1})(\gamma-q\delta)(D-\gamma)\notag\\
&\quad={}
 (r^2-q)D^2+q(D-\delta)^2+2(1-a_k)D\gamma\notag\\
&\qquad
 +\bigl[c_{k+1}^2
        -(1-a_{k+1})(1+\bar q+qa_{k+1})\bigr]\gamma^2.
\label{E:certificate-local-identity}
\end{align}
\end{lemma}

\begin{proof}
Fix $0\leq k<N$.  Expanding the left-hand side of
\eqref{E:certificate-local-identity} and collecting terms gives
\[
\begin{aligned}
&q(D-\delta)^2
 +\bigl[b_{k+1}^2-q(1-a_{k+1}^2)\bigr]D^2\\
&\quad
 +2\bigl[(1-a_{k+1})(1+qa_{k+1})
          -c_{k+1}b_{k+1}\bigr]D\gamma\\
&\quad
 +\bigl[c_{k+1}^2
        -(1-a_{k+1})(1+\bar q+qa_{k+1})\bigr]\gamma^2.
\end{aligned}
\]
By~\eqref{E:recurrence-interface-factorizations}, the coefficient of $D^2$
is $r^2-q$, while~\eqref{E:recurrence-boundary-product} identifies the
coefficient of $D\gamma$ as $2(1-a_k)$.  Substitution gives
\eqref{E:certificate-local-identity}.
\end{proof}

\begin{proof}[Proof of Lemma~\ref{L:certificate-decomposition}]
Put $r:=r_{\eta,w}$ and scale the certificate gap by setting
\[
    \Delta:=\bar q\bigl[\eta(w_xB_x+w_fB_f+w_gB_g)-\Phi\bigr].
\]
Using~\eqref{D:chain-quantities}, so that
\[
    \bar q\Phi=\bar qD_N^2-(D_N-\gamma_N)^2
          -q\sum_{\ell=0}^{N-1}(D_\ell-\delta_\ell)^2,
\]
and expanding $B_x,B_f,B_g$ from~\eqref{E:chain-quantities-explicit} gives
the $q$-sum and terminal terms below.
Expanding $(D_0-\gamma_0)^2$ in the $B_f$ contribution and collecting the
$D_k^2$ terms leaves the common coefficient
$\eta(\bar q w_x-qw_f)$ and the initial remainder
$\eta w_f(2D_0\gamma_0-\gamma_0^2)$.  By
\eqref{D:risk-coefficient-induced-radius}, the common coefficient is
$r^2-q$.  Hence
\begin{equation}
\label{E:certificate-gap-expanded}
\begin{aligned}
\Delta
&=(r^2-q)\sum_{k=0}^ND_k^2
 +q\sum_{k=0}^{N-1}(D_k-\delta_k)^2
 -\bar qD_N^2+(D_N-\gamma_N)^2\\
&\quad
 +\eta w_f(2D_0\gamma_0-\gamma_0^2)
 +\bar q\eta w_g\gamma_0^2.
\end{aligned}
\end{equation}

The last line of~\eqref{E:certificate-gap-expanded} is the only block
involving the initial weights.  By~\eqref{E:recurrence-interface-a0},
$\eta w_f=1-a_0-\bar qH_w(\eta)$.  Substituting this relation and
separating the matching-residual contribution gives
\[
\begin{aligned}
&\eta w_f(2D_0\gamma_0-\gamma_0^2)
 +\bar q\eta w_g\gamma_0^2\\
&\quad={}
 2(1-a_0)D_0\gamma_0
 +\bigl[\eta(\bar q w_g-w_f)-2\bar qH_w(\eta)\bigr]\gamma_0^2
 -2\bar qH_w(\eta)\gamma_0(D_0-\gamma_0).
\end{aligned}
\]
Replacing the last line of
\eqref{E:certificate-gap-expanded} by this expression gives
\[
\begin{aligned}
\Delta
&=(r^2-q)\sum_{k=0}^ND_k^2
 +q\sum_{k=0}^{N-1}(D_k-\delta_k)^2
 -\bar qD_N^2+(D_N-\gamma_N)^2\\
&\quad
 +2(1-a_0)D_0\gamma_0
 +\bigl[\eta(\bar q w_g-w_f)-2\bar qH_w(\eta)\bigr]\gamma_0^2\\
&\quad
 -2\bar qH_w(\eta)\gamma_0(D_0-\gamma_0).
\end{aligned}
\]

Starting from this expression, add and subtract the terms
$2(1-a_k)D_k\gamma_k$ for $1\leq k\leq N$.  Group the added terms with
the contribution having the same index and retain the subtracted terms as a
single cross-index remainder.  This gives
\[
\begin{aligned}
\Delta
&=\underbrace{\left[
\begin{aligned}
 &(r^2-q)D_0^2+q(D_0-\delta_0)^2
 +2(1-a_0)D_0\gamma_0\\
 &\quad
 +\bigl[\eta(\bar q w_g-w_f)-2\bar qH_w(\eta)\bigr]\gamma_0^2
\end{aligned}
\right]}_{\text{initial block}}\\
&\quad+\sum_{k=1}^{N-1}\underbrace{\left[
 (r^2-q)D_k^2+q(D_k-\delta_k)^2
 +2(1-a_k)D_k\gamma_k
\right]}_{\text{interior block at }k}\\
&\quad+\underbrace{\left[
\begin{aligned}
 &(r^2-q)D_N^2-\bar qD_N^2+(D_N-\gamma_N)^2\\
 &\quad+2(1-a_N)D_N\gamma_N
\end{aligned}
\right]}_{\text{terminal block}}\\
&\quad-2\sum_{k=1}^N(1-a_k)D_k\gamma_k
 -2\bar qH_w(\eta)\gamma_0(D_0-\gamma_0).
\end{aligned}
\]
For $1\leq k<N$,~\eqref{D:backward-coupling} gives
\[
    c_{k+1}^2=(1-a_{k+1})(1+\bar q+qa_{k+1}),
\]
so~\eqref{E:certificate-local-identity}, with
$(D,\gamma,\delta)=(D_k,\gamma_k,\delta_k)$, identifies each interior block.
At $k=0$,~\eqref{E:recurrence-interface-c1-square} gives
\[
    c_1^2-(1-a_1)(1+\bar q+qa_1)
      =\eta(\bar q w_g-w_f)-2\bar qH_w(\eta),
\]
so the same identity identifies the initial block.  Finally, since $a_N=r$
and $q+\bar q=1$, the terminal block satisfies
\[
\begin{aligned}
 &(r^2-q)D_N^2-\bar qD_N^2+(D_N-\gamma_N)^2
   +2(1-a_N)D_N\gamma_N\\
 &\qquad=(\gamma_N-a_ND_N)^2.
\end{aligned}
\]
Substituting these three evaluations into the preceding display yields
\[
\begin{aligned}
\Delta
&=\sum_{k=0}^{N-1}(b_{k+1}D_k-c_{k+1}\gamma_k)^2
 +(\gamma_N-a_ND_N)^2\\
&\quad
 +q\sum_{k=0}^{N-1}
  \bigl(\delta_k-a_{k+1}D_k-(1-a_{k+1})\gamma_k\bigr)^2\\
&\quad
 +2\sum_{k=0}^{N-1}(1-a_{k+1})
  (\gamma_k-q\delta_k)(D_k-\gamma_k)\\
&\quad
 -2\sum_{k=1}^N(1-a_k)D_k\gamma_k
 -2\bar qH_w(\eta)\gamma_0(D_0-\gamma_0).
\end{aligned}
\]
Reindexing the penultimate sum from $1,\ldots,N$ to $0,\ldots,N-1$ and
combining it with the preceding sum gives
\[
\begin{aligned}
\Delta
&=\sum_{k=0}^{N-1}(b_{k+1}D_k-c_{k+1}\gamma_k)^2
 +(\gamma_N-a_ND_N)^2\\
&\quad
 +q\sum_{k=0}^{N-1}
  \bigl(\delta_k-a_{k+1}D_k-(1-a_{k+1})\gamma_k\bigr)^2\\
&\quad
 +2\sum_{k=0}^{N-1}(1-a_{k+1})
  \Bigl[(\gamma_k-q\delta_k)(D_k-\gamma_k)
        -D_{k+1}\gamma_{k+1}\Bigr]\\
&\quad
 -2\bar qH_w(\eta)\gamma_0(D_0-\gamma_0).
\end{aligned}
\]
Dividing by $\bar q$ and using the definition of $\Delta$ gives
\eqref{E:certificate-decomposition}.
\end{proof}

\section{Specializations of the exact coefficient}
\label{A:specializations}

The convex boundary $q=0$, $w_x>0$ and the balanced slice considered below
admit explicit solutions of the scalar matching problem.  We derive them
directly from the radius relation~\eqref{D:risk-coefficient-induced-radius}
and matching residual~\eqref{E:mixed-matching-residual}, independently of the
continuation argument in Proposition~\ref{P:mixed-existence}.

\subsection{The convex boundary $q=0$}
\label{A:explicit-risk-coefficients}

\begin{proposition}[Convex weighted risk coefficient]
\label{P:convex-weighted-risk-coefficient}
Let $q=0$ and let $w$ satisfy the simplex conditions
\eqref{D:weight-simplex} with $w_x>0$.  Define
\begin{equation}\label{D:convex-weighted-theta-initial}
    \Theta_1(w):=
    \sqrt{\frac{w_x+w_f+\sqrt{w_x}}{2w_x}},
\end{equation}
and recursively
\[
    \Theta_{k+1}(w)
      :=\frac{1+\sqrt{1+4\Theta_k(w)^2}}2,
      \qquad k\geq1.
\]
Then, for every $N\geq1$, the scalar
\[
    \frac{4}{w_x\left(1+\sqrt{1+8\Theta_N(w)^2}\right)^2}
\]
is a risk coefficient for the weight $w$ at $q=0$.
\end{proposition}

\begin{proof}
Fix $N\geq1$, put
\[
    T:=\frac{1+\sqrt{1+8\Theta_N(w)^2}}2,
    \qquad
    r:=T^{-1},
    \qquad
    \eta:=\frac{1}{w_xT^2},
\]
and consider the paired backward sequence associated with $r$.  At $q=0$,
the second row of~\eqref{D:sequence} gives
\[
    b_k(r)=r,\qquad 1\leq k\leq N.
\]
We claim that
\[
    a_k(r)=1-2r^2\Theta_k(w)^2,
    \qquad 1\leq k\leq N.
\]
For $k=N$, this follows from
$T^2-T=2\Theta_N(w)^2$ and $a_N(r)=r$.  For the interior indices,
$\Theta_k(w)^2-\Theta_k(w)=\Theta_{k-1}(w)^2$, so the first row
of~\eqref{D:sequence} gives
\[
\begin{aligned}
    a_{k-1}(r)
      &=a_k(r)
        +\sqrt{2\bigl(1-a_k(r)\bigr)}\,b_k(r)\\
      &=a_k(r)+2r^2\Theta_k(w)
       =1-2r^2\Theta_{k-1}(w)^2.
\end{aligned}
\]
Thus the claimed formula holds by backward induction.
The recurrence gives $\Theta_{k+1}(w)>\Theta_k(w)>0$, while
$T>\Theta_N(w)$, so
$0<a_N(r)<\cdots<a_1(r)<1$ and $b_1(r)>0$.  Hence
$r\in\mathcal O_N(0)$.

Since $\eta=r^2/w_x$, the radius
relation~\eqref{D:risk-coefficient-induced-radius} gives
$r_{\eta,w}=r$.  Using
$a_1(r)=1-2r^2\Theta_1(w)^2$ and $b_1(r)=r$ in the matching
residual~\eqref{E:mixed-matching-residual} gives
\[
\begin{aligned}
    H_w(\eta)
      &=1-a_1(r)-\eta(w_x+w_f)-\sqrt{\eta}\,b_1(r)\\
      &=r^2\left(
          2\Theta_1(w)^2-
          \frac{w_x+w_f+\sqrt{w_x}}{w_x}
        \right)=0,
\end{aligned}
\]
where the last equality is~\eqref{D:convex-weighted-theta-initial}.
Thus $\eta$ is a risk coefficient for the weight $w$ at $q=0$.
\end{proof}

Once Theorem~\ref{T:intro-main} has been established, the proposition
identifies the unique risk coefficient as
\begin{equation}\label{E:convex-weighted-risk-coefficient}
    \eta_N^\star(0,w)
      =\frac{4}{w_x\left(1+\sqrt{1+8\Theta_N(w)^2}\right)^2}.
\end{equation}
Hence, for every $R>0$ and $d\geq2N+1$,
\begin{equation}\label{E:convex-weighted-exact-risk}
    \risk{N}{\mathcal P^{0,L}_{w,R}(\real^d)}
      =\frac{2LR^2}{w_x\left(1+\sqrt{1+8\Theta_N(w)^2}\right)^2}.
\end{equation}

At the distance vertex $w=w^x$, the convex weighted coefficient recovers the
standard OGM recurrence.  Indeed,
\eqref{D:convex-weighted-theta-initial} gives $\Theta_1(w^x)=1$.  Writing
$\theta_0:=1$ and $\theta_k:=\Theta_{k+1}(w^x)$ for $1\leq k<N$, and
$\theta_N:=(1+\sqrt{1+8\theta_{N-1}^2})/2$, the preceding formula gives
\begin{equation}\label{E:convex-distance-exact}
    \risk{N}{\finstances{d}}=\frac{L\Rx^2}{2\theta_N^2}.
\end{equation}
This recovers the exact convex value of~\cite{Drori2017Exact,KimFessler2016}.

\subsection{The balanced slice}
\label{A:sqrtq-specialization}

For $q\in(0,1)$, Lemma~\ref{L:sqrtq-membership} shows that the radius
$\sqrt q$ is always ordered.  The induced radius
\eqref{D:risk-coefficient-induced-radius} is independent of $\eta$ precisely
on the one-dimensional slice
\[
    w(u):=(qu,\bar q u,1-u),
    \qquad 0\leq u\leq1.
\]
Thus every weight on this slice uses the same paired sequence.  As on the
convex boundary above, this reduces the specialization to a scalar recurrence
and an explicit matching calculation.

\begin{proposition}[Balanced-slice risk coefficient]
\label{P:balanced-risk-coefficient}
Let $q\in(0,1)$ and $u\in[0,1]$, and write $\bar q=1-q$.  Define
\begin{equation}\label{D:sqrtq-initial-coordinate}
    \chi_1:=\sqrt{\frac{1+\sqrt q}{1-\sqrt q}},
\end{equation}
and recursively
\begin{equation}\label{E:sqrtq-coordinate-recurrence}
    \chi_{k+1}
      :=\frac{\chi_k}{1-\sqrt q}
        +\frac{\sqrt q}{\chi_k+\sqrt{\chi_k^2+\bar q}},
    \qquad k\geq1.
\end{equation}
In particular, $\chi_{k+1}>\chi_k/(1-\sqrt q)$, so the sequence grows at
least geometrically.  Then, for every $N\geq1$, the scalar
\[
    \frac{4\bar q}
      {\left(
          \sqrt q\,\chi_N
          +\sqrt{q\chi_N^2+2\bar q u(1+\chi_N^2)}
       \right)^2}
\]
is a risk coefficient for the weight $w(u)=(qu,\bar q u,1-u)$ at the given $q\in(0,1)$.
\end{proposition}

\begin{proof}
Fix $N\geq1$ and consider the paired backward sequence associated with
radius $\sqrt q$.  By Lemma~\ref{L:sqrtq-membership} it is ordered, and the invariant
\eqref{E:recurrence-interface-factorizations} gives
$b_k=\sqrt{q(1-a_k^2)}$.  We claim that
\begin{equation}\label{E:sqrtq-coordinate-map}
    a_{N+1-k}=\frac{\chi_k^2-1}{\chi_k^2+1},
    \qquad
    b_{N+1-k}=\frac{2\sqrt q\,\chi_k}{1+\chi_k^2},
    \qquad 1\leq k\leq N.
\end{equation}
The case $k=1$ follows from~\eqref{D:sqrtq-initial-coordinate}.  If the
first identity holds at index $k$, then substituting it and the invariant into
the first row of~\eqref{D:sequence} gives
\[
    \chi_{k+1}
      =\frac{\chi_k+\sqrt q\sqrt{\chi_k^2+\bar q}}{\bar q}.
\]
Using
\[
    \sqrt{\chi_k^2+\bar q}-\chi_k
      =\frac{\bar q}{\chi_k+\sqrt{\chi_k^2+\bar q}},
\]
this recurrence can be rewritten as
\eqref{E:sqrtq-coordinate-recurrence}; the claim follows by induction.

Set
\[
    \eta:=\frac{4\bar q}
      {\left(
          \sqrt q\,\chi_N
          +\sqrt{q\chi_N^2+2\bar q u(1+\chi_N^2)}
       \right)^2}.
\]
For $w=w(u)$, the radius relation gives $r_{\eta,w}=\sqrt q$.  Using~\eqref{E:sqrtq-coordinate-map} at $k=N$ in the
matching residual gives
\[
    H_{w(u)}(\eta)
      =\frac{2}{1+\chi_N^2}-u\eta
       -\frac{2\sqrt q\,\chi_N}{\sqrt{\bar q}(1+\chi_N^2)}\sqrt\eta.
\]
The chosen value of $\eta$ is the unique positive root of this expression,
so it is a risk coefficient for the weight $w(u)=(qu,\bar q u,1-u)$ at the given $q\in(0,1)$.
\end{proof}

Once Theorem~\ref{T:intro-main} has been established, the proposition
identifies the unique risk coefficient as
\begin{equation}\label{E:balanced-exact-coefficient}
    \eta_N^\star(q,w(u))
      =\frac{4\bar q}
      {\left(
          \sqrt q\,\chi_N
          +\sqrt{q\chi_N^2+2\bar q u(1+\chi_N^2)}
       \right)^2}.
\end{equation}
Hence, for every $R>0$ and $d\geq2N+1$,
\begin{equation}\label{E:balanced-exact-risk}
    \risk{N}{\mathcal P^{qL,L}_{w(u),R}(\real^d)}
      =\frac{2L\bar q R^2}
      {\left(
          \sqrt q\,\chi_N
          +\sqrt{q\chi_N^2+2\bar q u(1+\chi_N^2)}
       \right)^2}.
\end{equation}

\subsection{ITEM-f rate correspondence}
\label{A:itemf-correspondence}

On the initial function-value axis, the exact risk coefficient agrees with the
contraction factor of ITEM-f.  For the comparison, it is enough to recall the
scalar coefficient relations of ITEM-f.  To avoid collision with the paired
sequence, denote the ITEM-f coefficients called $(a_k,b_k)$ in
\cite[Lemma~4]{KimRyuDasGupta2026} by $(\alpha_k,\beta_k)$, let
$\Upsilon_N$ denote its contraction parameter, and put
$R_N:=\sqrt{\Upsilon_N^2-1}$.  The terminal relations are
\begin{equation}\label{E:gap-itemf-terminal-relations}
    \alpha_N=\Upsilon_N-\sqrt q\,R_N,
    \qquad
    \beta_N=\sqrt{\bar q}\,R_N.
\end{equation}
For $1\leq k<N$, define
\[
    \widehat d_k:=1-\frac{q\alpha_{k+1}}{\Upsilon_N},
    \qquad
    \widehat s_k:=\sqrt{1-\widehat d_k^2}.
\]
The ITEM-f coordinate relations can then be written in the rotation form
\begin{equation}\label{E:gap-itemf-coordinate-relations}
\begin{aligned}
    \alpha_k
      &=(\widehat d_k+q)\alpha_{k+1}-\widehat s_k\beta_{k+1},\\
    \beta_k
      &=\widehat s_k(\alpha_{k+1}-\Upsilon_N)
        +\widehat d_k\beta_{k+1}.
\end{aligned}
\end{equation}
These are equivalent to the coordinate relations of
\cite[Appendix~B.2]{KimRyuDasGupta2026}.  At the initial boundary, the same
construction gives
\begin{equation}\label{E:gap-itemf-boundary-relation}
    \beta_1
      =\sqrt{\frac{\bar q}{q}}\,(\Upsilon_N\alpha_1-1).
\end{equation}

\begin{proposition}[ITEM-f rate correspondence]
\label{P:gap-correspondence}
Let $N\geq1$ and $q\in(0,1)$.  Then
\begin{equation}\label{E:gap-itemf-risk}
    \eta_N^\star(q,w^f)=\Upsilon_N^{-2}.
\end{equation}
Consequently, the ITEM-f upper guarantee of
\cite[Section~4.2]{KimRyuDasGupta2026} is the exact finite-horizon minimax
risk on the initial function-value axis.
\end{proposition}

\begin{proof}
Set
\[
    \eta^\circ:=\Upsilon_N^{-2},
    \qquad
    r^\circ:=\sqrt{q(1-\eta^\circ)},
\]
and consider the paired backward sequence
$((a_k(r^\circ),b_k(r^\circ)))_{k=1}^N$.

\emph{Coefficient identification.}
We claim that, for $1\leq k\leq N$,
\begin{equation}\label{E:gap-itemf-coefficient-map}
    a_k(r^\circ)=1-\frac{\alpha_k}{\Upsilon_N},
    \qquad
    b_k(r^\circ)=\frac{\sqrt q\,\beta_k}{\Upsilon_N}.
\end{equation}
At $k=N$,~\eqref{E:gap-itemf-terminal-relations} gives
\[
    1-\frac{\alpha_N}{\Upsilon_N}
      =\frac{\sqrt q\,R_N}{\Upsilon_N}=r^\circ,
    \qquad
    \frac{\sqrt q\,\beta_N}{\Upsilon_N}
      =\sqrt{\bar q}\,r^\circ,
\]
so~\eqref{E:gap-itemf-coefficient-map} holds at the terminal index.  Now
suppose it holds at $k+1$.  The definitions of the two recurrences give
\[
    d_{k+1}=\widehat d_k,
    \qquad
    c_{k+1}=\frac{\widehat s_k}{\sqrt q}.
\]
Indeed, the first identity follows directly from
$a_{k+1}(r^\circ)=1-\alpha_{k+1}/\Upsilon_N$, and the second follows by
squaring and using~\eqref{D:backward-coupling}.  Substituting
\eqref{E:gap-itemf-coordinate-relations} into the rescaling
\eqref{E:gap-itemf-coefficient-map} therefore gives
\[
\begin{aligned}
    1-\frac{\alpha_k}{\Upsilon_N}
      &=d_{k+1}\left(1-\frac{\alpha_{k+1}}{\Upsilon_N}\right)
        +c_{k+1}\frac{\sqrt q\,\beta_{k+1}}{\Upsilon_N},\\
    \frac{\sqrt q\,\beta_k}{\Upsilon_N}
      &=-qc_{k+1}\left(1-\frac{\alpha_{k+1}}{\Upsilon_N}\right)
        +d_{k+1}\frac{\sqrt q\,\beta_{k+1}}{\Upsilon_N}.
\end{aligned}
\]
These are exactly the two rows of~\eqref{D:sequence}.  Backward induction
proves~\eqref{E:gap-itemf-coefficient-map}.  The order and positivity
properties in \cite[Lemma~4]{KimRyuDasGupta2026} therefore give
$r^\circ\in\mathcal O_N(q)$.

\emph{Matching.}
The boundary relation~\eqref{E:gap-itemf-boundary-relation}, together with
\eqref{E:gap-itemf-coefficient-map}, gives
\[
    \bar q\eta^\circ
      +\sqrt{\bar q\eta^\circ}\,b_1(r^\circ)
      =\frac{\bar q}{\Upsilon_N^2}
       +\frac{\sqrt{q\bar q}}{\Upsilon_N^2}\beta_1
      =\bar q\bigl(1-a_1(r^\circ)\bigr).
\]
Since $(r^\circ)^2=r_{\eta^\circ,w^f}^2$, division by $\bar q$ shows that
$\eta^\circ$ satisfies the matching equation for $w^f$.  Hence
$\eta^\circ$ is a risk coefficient for $w^f$, and uniqueness
proves~\eqref{E:gap-itemf-risk}.  The final assertion follows from
Theorem~\ref{T:intro-main}.
\end{proof}

\bibliographystyle{abbrvnat}
\bibliography{ITEM-exact-v1}

\end{document}